\documentclass{amsart}

\usepackage{custom_dirk}
\usepackage{amsfonts}
\usepackage{url}
\usepackage{tabularx}
\usepackage{vmargin}
\usepackage{longtable}
\usepackage{tikz}

\usepackage[giveninits=true,maxnames = 10,backend=biber]{biblatex}
\AtEveryBibitem{
  \clearfield{issn} 
  \clearfield{doi} 
  \clearlist{language}
  \clearfield{isbn}

  \ifentrytype{online}{}{
    \ifentrytype{manual}{}{\clearfield{url}}
  }
}

\bibliography{references}

\usepackage{changepage}

\newcommand{\GW}{\text{GW}}

\usepackage[hidelinks]{hyperref}

\begin{document}

\title{Virasoro constraints for moduli spaces of sheaves on surfaces}

\author{Dirk van Bree}
\date{September 2021}

\subjclass[2020]{14J60; 14M25}

\begin{abstract}
  We introduce a conjecture on Virasoro constraints for the moduli space of stable sheaves on a smooth projective surface. These generalise the Virasoro constraints on the Hilbert scheme of a surface found by Moreira and Moreira, Oblomkov, Okounkov and Pandharipande. We verify the conjecture in many nontrivial cases by using a combinatorial description of equivariant sheaves found by Klyachko.
\end{abstract}

\maketitle

\section{Introduction}
\label{sec:introduction}

The Virasoro constraints are a conjecture in Gromov-Witten theory, proposed by Eguchi, Hori and Xiong \cite{EGUCHI199771} for any smooth projective complex variety $Y$ with only $(p,p)$-cohomology. S. Katz helped to establish a general conjecture. We review the conjecture below. The Virasoro constraints have been proven if $Y$ is a toric 3-fold, see \cite{Givental:2001aja}. Recently, Moreira, Oblomkov, Okounkov and Pandharipande \cite{MOOP_unpublished} used the GW-PT correspondence \cite{pandharipande12_descen_local_curves, pandharipande12:_descen, pandharipande13_descen_theor_stabl_pairs_toric, 10.2140/gt.2014.18.2747} to obtain constraints on the moduli space of stable pairs of a toric 3-fold. In \cite{MOOP_unpublished} and \cite{moreira20_viras_conjec_stabl_pairs_descen_unpublished}, the theory was applied to the case where the toric 3-fold is $X \times \mathbb{P}^1$, for $X$ a toric surface. In the second reference, a cobordism argument is used to prove constraints for the Hilbert scheme of points of $X$, for any simply connected surface $X$. In this paper, we conjecture Virasoro constraints extending this result to moduli spaces of sheaves of rank $r \geq 1$. In doing so, we address a question of R. Pandharipande, who asked in his Hangzhou lecture on Virasoro constraints (April 2020) whether the moduli space of stable sheaves admits such constraints.

\subsection{Virasoro constraints in GW-theory} We review the Virasoro conjecture for GW-theory (also see \cite[Sec. 4]{pandharipande_three_questions_gro_wit_thy}). Let $Y$ be a smooth projective complex variety. Assume for simplicity that $Y$ has only $(p,p)$-cohomology. For any nonnegative numbers $n$, $g$ and any $\beta \in H_2(Y, \mathbb{Z})$ we associate to $Y$ the moduli space of stable maps $\bar{M}_{g, n}(Y, \beta)$. Recall that a stable map is a morphism $f : C \to Y$ where $C$ is a connected genus $g$ curve with at most nodal singularities and $n$ marked smooth points $x_1, x_2, \ldots, x_n$. Furthermore, $f$ should satisfy $f_*[C] = \beta$. Then there are canonical maps $\ev_i : \bar{M}_{g,n}(Y, \beta) \to Y$ which sends a stable map $(C, f)$ to $f(x_i)$, the image of the $i$-th marked point. Also, $\bar{M}_{g,n}(Y, \beta)$ admits $n$ canonical line bundles $M_i$ for $1 \leq i \leq n$. At the point $(C, f)$, $M_i$ is the cotangent bundle to $C$ at $x_i$. Then we define the descendant GW-invariants as
\[
  \langle \tau_{k_1}(\gamma_1) \tau_{k_2}(\gamma_2) \ldots \tau_{k_n}(\gamma_n)\rangle^Y_{g,\beta} = \int_{[\bar{M}_{g,n}(Y, \beta)]^{\text{vir}}} c_1(M_1)^{k_1}\ev_1^*(\gamma_1) \cdots c_1(M_n)^{k_n} \ev_n^*(\gamma_n)
\]
Now, fix a basis $\{\gamma_i\}_{i = 1,\ldots,r}$ of $H^\bullet(Y, \mathbb{Q})$, which is homogeneous with respect to the degree decomposition. Let $\lambda$, $q$ and $\{t^a_k\}_{k = 0,1,\ldots}^{a = 1,\ldots,r}$ be formal variables. Then we define the GW-partition function as
\[
  F^Y = \sum_{g\geq 0} \lambda^{2g-2} \sum_{\beta \in H_2(Y, \mathbb{Z})} q^\beta \sum_{n \geq 0} \frac{1}{n!} \sum_{\substack{a_1,\ldots,a_n \\ k_1,\ldots,k_n}} t^{a_1}_{k_1}\cdots t^{a_n}_{k_n} \langle  \tau_{k_1}(\gamma_{a_1}) \ldots \tau_{k_n}(\gamma_{a_n}) \rangle^X_{g, \beta}
\]

Let $Z^Y = \exp(F^Y)$. The Virasoro conjecture states that $L^{\GW}_k(Z^Y) = 0$, where $L^{\GW}_k$ are certain formal differential operators defined for $k\geq -1$. We will not give the formulae for the $L_k^{\GW}$ here, see \cite[Sec. 4]{pandharipande_three_questions_gro_wit_thy} for those. The operators $L_k^{\GW}$ are called the \emph{Virasoro operators} and they satisfy the \emph{Virasoro bracket}
\begin{equation}\label{eq:virasoro-bracket}
  [L^{\GW}_k, L^{\GW}_n] = (n - k)L^{\GW}_{k+n}.
\end{equation}

The Virasoro conjecture expresses relations among integrals on $\bar{M}_{g, n}(X, \beta)$, for various $g$, $n$ and $\beta$. It is possible to write down these relations rather explicitely. Let $\mathbb{D}^Y_{\text{GW}}$ be the commutative $\mathbb{Q}$-algebra generated by formal symbols $\tau_k(\gamma_i)$. Then one can define certain operators $\mathcal{L}_k^\GW$ on this algebra for $k \geq -1$. These satisfy the Virasoro bracket. For each $g$ and $\beta$ we have a linear map $\langle - \rangle^Y_{g, \beta}:\mathbb{D}^Y_{\text{GW}} \to \mathbb{Q}$, which sends the monomial $\tau_{k_1}(\gamma_{a_1})\cdots \tau_{k_n}(\gamma_{a_n})$ to $\langle \tau_{k_1}(\gamma_{a_1}) \ldots \tau_{k_n}(\gamma_{a_n}) \rangle^Y_{g, \beta}$. The Virasoro constraints described above are equivalent to the vanishing
\[
  \langle \mathcal{L}^{\GW}_kD \rangle^Y_{g, \beta} = 0
\]
for any $D \in \mathbb{D}^Y_\GW$ and all $g$ and $\beta$. This formulation of the Virasoro constraints generalises to other contexts, as we will see now.

\subsection{Virasoro constraints for stable pairs}

Recall that a stable pair on a smooth projective threefold $Y$ is a map of coherent sheaves $s : \mathcal{O}_Y \to F$ such that $F$ is pure of dimension 1 (i.e., every nonzero subsheaf of $F$ has dimension 1) and $\dim \Supp \coker s = 0$. Associated to such a pair are two invariants, $n = \chi(Y, F)$ and $\beta \in H_2(Y, \mathbb{Z})$, the homology class associated to $\Supp F$. There is a fine projective moduli space $P = P_n(Y, \beta)$ parametrising stable pairs with these invariants. It carries a virtual fundamental class and its virtual dimension is $\int_{\beta}c_1(X)$. See \cite{pandharipande09_curve_count_via_stabl_pairs_deriv_categ} for more details.

The GW-PT correspondence describes a relationship between the GW- and PT-invariants of a smooth quasi-projective 3-fold. The so-called stationary variant has been proven in the toric case \cite{10.2140/gt.2014.18.2747}. In \cite{MOOP_unpublished}, this correspondence has been made more explicit and it is used to derive constraints for the moduli space of stable pairs on $Y$.

Define $\mathbb{D}^Y$ to be the $\mathbb{Q}$-algebra generated by formal symbols $\ch_i(\gamma)$, where $\gamma \in H^\bullet(Y, \mathbb{Q})$. We impose the relations $\ch_i(\gamma + \gamma') = \ch_i(\gamma) + \ch_i(\gamma')$ and $\ch_i(\lambda\gamma) = \lambda\ch_i(\gamma)$. Again it is possible to define operators $\mathcal{L}_k^{\text{PT}}$ on this algebra. These satisfy the Virasoro bracket in a slightly weaker sense. Fix a number $n$ and a class $\beta \in H_2(Y, \mathbb{Z})$. Let $P = P_n(Y, \beta)$ be the moduli space of stable pairs. Then there is an algebra homomorphism $\mathbb{D}^Y \to H^\bullet(P, \mathbb{Q})$ by sending $\ch_i(\gamma)$ to
\[
  \pi_{P, *}(\ch_i(\mathbb{F} - \mathcal{O}_{Y \times P}) \cdot \pi^*_Y(\gamma))
\]
Here $\pi_P$ and $\pi_Y$ are the projection from $Y \times P$ to $P$ and $Y$, respectively, and $\mathcal{O}_{Y \times P} \to \mathbb{F}$ is the universal stable pair on $Y \times P$. The Virasoro constraints obtained by \cite{MOOP_unpublished} can now be expressed as
\[
  \int_{[P_n(Y, \beta)]^{\text{vir}}} \mathcal{L}_k^{\text{PT}}D = 0
\]
for all $D \in \mathbb{D}^Y$ and all $k \geq -1$. (Here the map to the cohomology of $P$ is left implicit.) These constraints hace been proven for $Y$ toric, if $D$ is contained in a certain subalgebra of $\mathbb{D}^Y$ consisting of the \emph{stationary invariants} (see \cite[Thm. 4]{MOOP_unpublished}).

Note that in contrast to the Gromov-Witten case, here we have relations between integrals on a \emph{single} moduli space. We did not write out the formula for $\mathcal{L}_k^{\text{PT}}$, but it is very similar to our Definition \ref{def:virasoro-operator} and the reader is encouraged to compare our definition with Definition 2 in \cite{MOOP_unpublished}.

\subsection{Virasoro constraints for the Hilbert scheme}

In \cite[Sec. 6]{MOOP_unpublished} they take $Y = X \times \mathbb{P}^1$, for $X$ a toric surface. If one takes any $x \in X$, then the moduli space of pairs $P_n(X \times \mathbb{P}^1, n[\{x\} \times \mathbb{P}^1])$ is canonically isomorphic to $\Hilb^n(X)$. The Virasoro constraints of pairs induce constraints on $\Hilb^n(X)$. We have the algebra $\mathbb{D}^X$, which is defined in the same way as before. We have an operator $\mathcal{L}_k$, which is the same as the one from Definition \ref{def:virasoro-operator}. We also have a ring homomorphism $\mathbb{D}^X \to H^\bullet(\Hilb^n(X), \mathbb{Q})$, which is defined by sending $\ch_i(\gamma)$ to
\[
  \pi_{\Hilb^n(X), *}(-\ch_k(\mathcal{I}) \cdot \pi_X^*(\gamma))
\]
where $\pi_{\Hilb_n(X)}$ and $\pi_X$ are the projections from $X \times \Hilb^n(X)$ to $\Hilb^n(X)$ and $X$ respectively, and $\mathcal{I}$ is the universal ideal sheaf. Then the Virasoro constraints are
\[
  \int_{\Hilb^n(X)} \mathcal{L}_kD = 0
\]
for all $D \in \mathbb{D}^X$ and all $k \geq -1$. This was proven for all smooth projective toric surfaces $X$ in \cite[Sec. 6]{MOOP_unpublished}. In \cite{moreira20_viras_conjec_stabl_pairs_descen_unpublished}, a similar formula is proven for all simply-connected smooth projective surfaces $X$, using a cobordism argument. This required a modification of the operator $\mathcal{L}_k$, because the formula of Definition \ref{def:virasoro-operator} is only correct if $X$ has only $(p,p)$-cohomology.

\subsection{Formulation of the conjecture}

We will now formulate the main conjecture of this paper. Let $X$ be a surface\footnote{In this paper, this will mean a smooth projective complex variety of dimension two.} which only has $(p,p)$-cohomology and fixed polarisation $H$. Then $H^\bullet(X, \mathbb{Q})= \bigoplus_{i = 0}^2 H^{2i}(X, \mathbb{Q})$ and we refer to the elements of $H^{2i}(X, \mathbb{Q})$ as the cohomology classes of (complex) degree $i$. We fix in advance integers $r > 0$ and $c_2$ and a line bundle $\Delta$ on $X$ and we let $M = M^H_X(r, \Delta, c_2)$ be the moduli space of Gieseker semi-stable sheaves (with respect to $H$) with rank $r$, determinant $\Delta$ and second Chern class $c_2$.

\begin{definition}
  We define $\mathbb{D}^X$ as the commutative $\mathbb{Q}$-algebra generated by symbols of the form $\ch_i(\gamma)$, where $i$ is a nonnegative integer and $\gamma \in H^\bullet(X, \mathbb{Q})$. We impose the relations $\ch_i(\gamma_1 + \gamma_2) = \ch_i(\gamma_1) + \ch_i(\gamma_2)$ and $\ch_i(\lambda \cdot \gamma) = \lambda \cdot \ch_i(\gamma)$. We define a grading on $\mathbb{D}^X$ by declaring that the degree of $\ch_i(\gamma)$ is $i + \deg \gamma - 2$. 
\end{definition}

Later we will interpret the $\ch_i(\gamma)$'s as elements of the cohomology on $M$. The degree of $\ch_i(\gamma)$ is chosen so that it matches the degree of the cohomology class we will associate to it. We introduce $\mathbb{D}^X$ because the operator $\mathcal{L}_k$ defined below does not descend to the level of cohomology. The next definition extends the one given in \cite{MOOP_unpublished} and \cite{moreira20_viras_conjec_stabl_pairs_descen_unpublished}\footnote{Up to a typo in the second reference.}.

\begin{definition}\label{def:virasoro-operator}
  For each $k \geq -1$, we define an operator $\mathcal{L}_k$ on $\mathbb{D}^X$ as $R_k + T_k + S_k$ where the latter three operators are:
  \begin{itemize}
  \item  $R_k$ is defined by $R_k\ch_i(\gamma) = \prod_{j = 0}^k (i + j + d - 2) \ch_{i + k}(\gamma)$ for $\gamma \in H^\bullet(X, \mathbb{Q})$ of degree $d$. We then define it on all of $\mathbb{D}^X$ by requiring it to be a derivation. In particular, $R_{-1}\ch_i(\gamma) = \ch_{i - 1}(\gamma)$, where we agree that $\ch_{-1}(\gamma) = 0$.
  \item
    $T_k$ is multiplication by a fixed element of $\mathbb{D}^X$, namely
    \begin{align*}
      T_k = & - \sum_{a + b = k + 2} (-1)^{(d^L + 1)(d^R + 1)}(a + d^L - 2)!(b + d^R - 2)!\ch_a\ch_b(1)\\
      &+ \sum_{a + b = k} a!b!\ch_a\ch_b\left(\frac{c_1(X)^2 + c_2(X)}{12}\right)
    \end{align*}
    Here we are using
    \[(-1)^{(d^L + 1)(d^R + 1)}(a + d^L - 2)!(b + d^R - 2)!\ch_a\ch_b(1)\]
    as an abbreviation for
    \[\sum_i (-1)^{(\deg(\gamma_i^L) + 1)(\deg(\gamma_i^R) + 1)}(a + \deg(\gamma_i^L) - 2)!(b + \deg(\gamma_i^R) - 2)!\ch_a(\gamma_i^L)\ch_b(\gamma_i^R)\]
    where $\sum_i \gamma_i^L \otimes \gamma_i^R$ is the Künneth decomposition of $\Delta_*1 \in H^4(X \times X,\mathbb{Q})$. In the second sum $\ch_a\ch_b\left(\frac{c_1(X)^2 + c_2(X)}{12}\right)$ is a similar abbreviation. Note that on $X$ we have an equality $\frac{c_1(X)^2 + c_2(X)}{12} = \chi(X, \mathcal{O}_X) \cdot \mathbf{p}$ by Hirzebruch-Riemann-Roch, where $\mathbf{p}$ is the point class. Finally, we adopt the convention that factorials of negative numbers are zero.
  \item
    $S_k$ is defined by setting $S_kD = \frac{(k+1)!}{r} R_{-1}(\ch_{k+1}(\mathbf{p})D)$. Here $\mathbf{p} \in H^4(X, \mathbb{Q})$ again corresponds to the class of a point.
  \end{itemize}
\end{definition}

Note that the definition of $S_k$ depends on the rank $r$, but the definition of $R_k$ and $T_k$ does not. A first thing to notice is that $\mathcal{L}_k$ is of degree $k$. This is easily verified, as this is true for $R_k$, $T_k$ and $S_k$ separately. Next one might wonder whether some part of the operator satisfies the Virasoro bracket. This is almost true. In section \ref{sec:notat-first-evid} where we prove that $R_k + T_k$ satisfies the Virasoro bracket after a natural modification of the above definition.

Next we explain how to interpret these as cohomology classes on $M = M^H_X(r, \Delta, c_2)$. Assume that $M$ is fine, i.e., that there is a universal sheaf $\mathcal{E}$ on $X \times M$. Now consider the following cohomology classes on $X \times M$:
\begin{equation}
 \label{eq:kernel-classes} 
  \ch_i\left(- \mathcal{E} \otimes \left(\det{\mathcal{E}}\right)^{-1/r}\right).
\end{equation}

Of course $\left(\det \mathcal{E}\right)^{-1/r}$ might not exist as a line bundle. Then the above cohomology class is still defined as the degree $i$ part of $-\ch(\mathcal{E}) \cdot \ch(\det(\mathcal{E}))^{-1/r}$, a formal power series in the cohomology ring. In general the universal sheaf is not unique, it is determined up to tensoring with a line bundle pulled back from $M$. However, this does not change the above classes, so these are canonically associated to $M$. Lastly, the existence of a universal family is not needed to construct the above class, as we will explain in Section \ref{sec:notat-first-evid}.

This construction allows us to interpret the $\ch_i(\gamma)$ as cohomology classes on $M$ by means of a slant product. Consider the projections $\pi_X : X \times M \to X$ and $\pi_M : X \times M \to M$.

\begin{definition}
  We define the \emph{geometric realisation} of a formal symbol $\ch_i(\gamma)$ as
  \[
    \ch_i(\gamma) = \pi_{M, *}\left( \pi_X^*\gamma \cdot \ch_i(-\mathcal{E} \otimes \det(\mathcal{E})^{-1/r})\right)
  \]
  This gives an algebra homomorphism $\mathbb{D}^X \to H^\bullet(M, \mathbb{Q})$. 
  
\end{definition}

We use the same notation for the elements of $\mathbb{D}^X$ and their geometric realisations. This will not cause confusion as long as one remembers that the operators $\mathcal{L}_k$ operate only on the formal algebra $\mathbb{D}^X$ and not on the cohomology of $M$. Also note that the geometric realisation is degree-preserving. 

Now we formulate the conjecture in a fairly general setting. Recall that a virtual fundamental class for $M$ was constructed by T. Mochizuki \cite{mochizuki09_donal} for the moduli space of stable sheaves. 

\begin{conjecture}
\label{sec:strong-conjecture}
Let $X$ be a surface with only $(p,p)$-cohomology and fixed polarisation $H$. Choose numbers $r > 0$ and $c_2$ and a line bunde $\Delta$. Let $M = M^H_X(r,\Delta, c_2)$ be the moduli space of Gieseker semi-stable sheaves of rank $r$, with determinant $\Delta$ and second Chern class $c_2$. Assume that all semi-stable sheaves with these invariants are stable. Then for all integers $k \geq -1$ and all $D \in \mathbb{D}^X$ we have
  \[
    \int_{[M]^{\text{vir}}} \mathcal{L}_k D = 0
  \]
\end{conjecture}

In most of the evidence presented in this paper, $M$ is smooth of the expected dimension. (However, we have some general statements in virtual setting in Section \ref{sec:notat-first-evid}.) Hence the virtual integral simplifies to an ordinary integral in those cases. The case where $r = 1$, $\Delta = \mathcal{O}_X$ (so $M$ is a Hilbert scheme of points) has already been proven \cite[Thm. 5]{moreira20_viras_conjec_stabl_pairs_descen_unpublished}, under the additional assumption that $X$ is simply-connected.

We will provide plenty additional evidence for the conjecure. For $k = -1$ and $k = 0$ one can verify the conjecture directly, see Prop. \ref{prp:k-0--1}. The fact that $\mathcal{L}_k$ is of degree $k$ will imply that the conjecture holds for $k > \vdim M$. The remaining evidence is a collection of explicit calculations of certain moduli spaces of sheaves on toric surfaces. 

We explain what evidence we have. In the calculations we assume that $\gcd(r, \Delta.H) = 1$. This implies that the moduli space is fine and that Gieseker stability coincides with $\mu$-stability. Recall also that by the Bogomolov inequality, for any fixed $X$, $r$ and $\Delta$, there is a minimal $c_2$ such that $M$ is nonempty (see \cite[Thm. 3.4.1]{huybrechts10}). In all the cases in this paper, this minimal $c_2$ coincides with the smallest number $c_2$ such that the Bogomolov inequality is satisfied. From the proof one can infer that for this minimal $c_2$, $M$ consists only of vector bundles. With this in mind, we have verified the conjecture in the following cases:
\begin{itemize}
\item $X = \mathbb{P}^2$, $c_1 = 1$, for $r = 2$, $3$ and $4$ with the minimal $c_2$ (which is respectively $1$, $2$ and $3$). Note that the choice of polarisation on $X$ is irrelevant.\footnote{In the case $r = 4$ there is a technical difficulty, see section \ref{sec:results-toric-vari}.}
\item $X = \mathbb{P}^2$, $r = 2$, $c_1 = 1$ and $c_2 = 2,3$. In these calculations, non-locally-free sheaves appear.
\item $X = \mathbb{F}_a$, the Hirzebruch surface, with $r = 2$, for any polarisation $H$ such that $H.\Delta$ is odd, and $c_2$ minimal.
\item $X = \mathbb{F}_0 = \mathbb{P}^1 \times \mathbb{P}^1$, with  $\Delta = \{*\} \times \mathbb{P}^1$ and $c_2 = 2$. Here the minimal $c_2$ is 1. Here we have taken $H$ to be an arbitrary polarisation such that $H.\Delta$ is odd.
\end{itemize}

In these cases, the dimension of the moduli space ranges from $0$ to $8$. We verify the conjecture by verifying it for monomials in the $\ch_i(\gamma)$'s. Thus the number of independent checks is equal to the dimension the vector space of such monomials. This dimension varies greatly from case to case, the largest dimension we encountered being 993.

\begin{example}
  The innocent-looking $\mathcal{L}_{\dim M}1$ is already non-trivial. Note that $R_{\dim M}1 = 0$ in all cases. Consider the case $X = \mathbb{P}^2$ and $r = 4$ mentioned above. In this case $\dim M = 6$. Then
  \[
    \int_MT_{6}1 = - \frac{49511}{4096} \qquad \text{and} \qquad \int_MS_{6}1 = \frac{49511}{4096}.
  \]
  Hence the conjecture holds in this case.
\end{example}

\begin{example}
  We construct a more complicated example. Consider again $X = \mathbb{P}^2$ and $r = 4$, just as before. Let $k = 2$ and $D = \ch_2(\mathbf{p})\cdot \ch_3(1)^2$. Then
  \[
    \int_MR_2D = -\frac{29715}{16348}\quad,\quad \int_MT_2D = \frac{18825}{32768} \quad \text{and} \quad \int_MS_2D = \frac{40605}{32768}.
  \]
  Again, these sum up to zero, as required.
\end{example}

See Appendix \ref{sec:data-obtained-from} for more explicit numbers obtained from the calculations. In total we did 1677 independent checks. The number of checks grows quickly with the dimension of the moduli space, as the 993 independent checks mentioned above came from the case $X = \mathbb{P}^2$, $r = 2$, $c_1 = 1$ and $c_2 =  3$, where the moduli space had dimension 8.

The strategy to verify the conjecture in these explicit cases comes from toric geometry. Since $X$ admits a toric action, so does the moduli space $M$. The fixed-point locus admits an explicit combinatorial description, due to Klyachko \cite{Klyachko_Equiv_bundle_toral_var}, Perling \cite{Perling_2004} and Kool \cite{KOOL20111700}. We review this description in Section \ref{sec:invariant-sheaves}. Then we apply Atiyah-Bott localisation to evaluate the integral. In the cases we consider, the fixed-point locus is always isolated, but the results are still interesting and nontrivial.

\subsection{Possible variations of the conjecture}

In Conjecture \ref{sec:strong-conjecture} we required that $X$ only has $(p,p)$-cohomology. In the Hilbert scheme case, Moreira \cite{moreira20_viras_conjec_stabl_pairs_descen_unpublished} has been able to remove this assumption at the cost of the operators $R_k$, $T_k$ and $S_k$ becoming more complicated. We expect that a similar modification can be made in the surface case, but it is unclear to the author if exactly the same modification works.

We have also assumed that all semi-stable sheaves with our invariants are stable. This assumption is needed to have a virtual fundamental class on $M$, and is thus indispensible. It is worth investigating if there is a version of Conjecture \ref{sec:strong-conjecture} on a different space, such as the space of Bradlow pairs or Joyce-Song pairs \cite{mochizuki09_donal} \cite{Joyce_vertex_algs}.

\subsection{Acknowledgements.} The author would like to thank his supervisor Martijn Kool for many useful suggestions and learning from him many of the techniques employed in this paper. He is also grateful to Rahul Pandharipande for raising the problem addressed in this paper in his Hangzhou lecture as well as provding some references. Finally, he would like to thank Sergej Monavari, Woonam Lim, Longting Wu and Carel Faber for insightful comments. The author is supported by NWO grant VI.Vidi.192.012.

\section{First remarks}
\label{sec:notat-first-evid}

\subsection{Eliminating fineness} The cohomology classes of~(\ref{eq:kernel-classes}) were constructed using a universal family of semi-stable sheaves on $X \times M$. This universal family is not needed, since we have always access to a \emph{twisted} universal family $\mathcal{E}$ \cite{căldăraru2000derived}. Denote its Brauer class by $\alpha$. Then $\mathcal{E}^{\otimes r}$ and $\det\mathcal{E}$ both have Brauer class $r\alpha$, where $r$ is the rank of $\mathcal{E}$. In particular $\mathcal{E}^{\otimes r} \otimes \left( \det{\mathcal{E}} \right)^{-1}$ has Brauer class 0, i.e., it is an ordinary sheaf. Therefore, we might compute the cohomology classes above by taking the Chern classes of this sheaf and then taking the $r$-th root on the level of cohomology. Finally, we note that this is independent of the twisted family chosen. Indeed, if $\mathcal{E}' = \mathcal{E} \otimes L$ is another family, for $L$ a line bundle, then $\mathcal{E}'^{\otimes r} = L^r \otimes \mathcal{E}^ {\otimes r}$ and $\det\mathcal{E}' = L^r \det\mathcal{E}$, hence $\mathcal{E}'^{\otimes r} \otimes \det \mathcal{E}' \cong \mathcal{E}^{\otimes r} \otimes \det \mathcal{E}$.

\subsection{The conjecture for small \texorpdfstring{$k$}{k}}

For small $k$, it is actually possible to verify the conjecture by providing identities for $\ch_0(\gamma)$ and $\ch_1(\gamma)$. These equations are similar to Proposition 1 of Moreira \cite{moreira20_viras_conjec_stabl_pairs_descen_unpublished}. 

\begin{lemma}
  \label{lmm:ch0-ch1-identities}
  For any smooth surface $X$, $r > 0$ and Chern classes $c$, we have the following identities in the cohomology of $M$:
  \begin{enumerate}
  \item $\ch_0(\gamma) = -r\cdot\int_X\gamma \in H^ 0(M, \mathbb{Q})$.
  \item $\ch_1(\gamma) = 0.$
  \end{enumerate}
\end{lemma}
\begin{proof} Let $I = H^{\geq 4}(X \times M, \mathbb{Q})$. We can write $\ch(\mathcal{E})$ as $r + c_1(\mathcal{E}) \mod I$. Similarly, we can write $\ch(\det(\mathcal{E})^ {-1/r})$ as $e^ {-c_1(\mathcal{E})/r} = 1 - c_1(\mathcal{E})/r \mod I$. Their product is $r \mod I$. We obtain that $\ch_k(-\mathcal{E} \otimes \det\mathcal{E}^{-1/r})$ is $-r$ for $k = 0$ and $0$ for $k = 1$. This second identity implies $\ch_1(\gamma) = 0$. If we use the push-pull formula, then the first identity implies
  \[
    \ch_0(\gamma) = \pi_{M, *}(\pi_X^ *\gamma \cdot -r) = -r \pi_{M, *}\pi_X^*\gamma = -r \cdot\int_X\gamma \in H^0(M,\mathbb{Q}).
  \]
\end{proof}

One minor annoyance is that the algebra $\mathbb{D}^X$ contains elements of negative degree, e.g. $\ch_0(1)$ has degree $-2$. We will deal with these elements in the next lemma, telling us that we can essentially ignore these.

\begin{lemma}
  \label{lmm:low-term-lemma}
  Let $\gamma \in H^\bullet(X, \mathbb{Q})$ be an element of pure degree. If $\deg \ch_i(\gamma) \leq 0$, then for all $D \in \mathbb{D}^X$ we have $\mathcal{L}_k(\ch_i(\gamma)D) = \ch_i(\gamma)\mathcal{L}_kD$ in $H^\bullet(M, \mathbb{Q})$.
\end{lemma}
\begin{proof}
  We verify this for $R_k$, $T_k$ and $S_k$ separately. For $T_k$ is it immediate. For $R_k$ we note that $R_k(\ch_i(\gamma)D) = R_k(\ch_i(\gamma))D + \ch_i(\gamma)R_kD$. But note that $R_k(\ch_i(\gamma)) = 0$ because either this has negative degree, or there is a zero in the product in the definition of $R_k$.\\
  For $S_k$, we use again that $R_{-1}$ is a derivation to see that
  \[
    S_k(\ch_i(D)) = \frac{(k+1)!}{r}\left(\ch_{k+1}(\mathbf{p})DR_{-1}\ch_i(\gamma) + \ch_{i}(\gamma)R_{-1}\left(\ch_{k+1}(\mathbf{p})D\right)\right)
  \]
  But $R_{-1}(\ch_i(\gamma))$ has strictly negative degree, so it vanishes in the cohomology of $M$.
\end{proof}

\begin{corollary}
  \label{sec:large-k}
  Conjecture  \ref{sec:strong-conjecture} holds when $k > \vdim M$.
\end{corollary}

\begin{proof}
  It suffices to check the conjectures when $D = \prod_j \ch_{i_j}(\gamma_j)$ where the $\gamma_j$ are of pure degree. Then the degree of $D$ is $\sum_j \deg \ch_{i_j}(\gamma_j)$. If $\deg D \geq 0$ then $\deg \mathcal{L}_kD > \dim M$, so the integral is zero. If $\deg D < 0$ assume $\deg \ch_{i_0}(\gamma_0) < 0$. By Lemma \ref{lmm:low-term-lemma} we find that $\mathcal{L}_kD$ is a multiple of $\ch_{i_0}(\gamma_0) = 0$, so it is zero as well.
\end{proof}

\begin{corollary}
  \label{cor:eliminiating-ch-1}
  We have that $\mathcal{L}_k(\ch_1(\mathbf{p})D) = 0$, for any $D \in \mathbb{D}^X$.
\end{corollary}
\begin{proof}
  Lemma \ref{lmm:ch0-ch1-identities} tells us that $\ch_1(\mathbf{p}) = 0$ in $H^\bullet(M, \mathbf{Q})$, hence $T_k(\ch_1(\mathbf{p})) = 0$. By the same lemma,
  \[
    R_k(\ch_1(\mathbf{p})D) = R_k(\ch_1(\mathbf{p}))D = (k+1)!\ch_{k+1}(\mathbf{p})D
  \]
  and
  \[
    S_k(\ch_1(\mathbf{p})D) = \frac{(k+1)!}{r}R_{-1}(\ch_{k+1}(\mathbf{p})\ch_1(\mathbf{p})D) = \frac{(k+1)!}{r}\ch_{k+1}(\mathbf{p})\ch_0(\mathbf{p})D.
  \]
  Now $\ch_0(\mathbf{p}) = -r$ by the other identity from Lemma \ref{lmm:ch0-ch1-identities}. So we are done.
\end{proof}

\begin{proposition}
  \label{prp:k-0--1}
  Conjecture \ref{sec:strong-conjecture} hold when $k = -1$ or $k = 0$.
\end{proposition}

\begin{proof}
  For $k = -1$, note that
  \[
    S_{-1}D = \frac1rR_{-1}(\ch_0(\mathbf{p})D) = \frac1r \ch_{-1}(\mathbf{p})D + \frac1r \ch_0(\mathbf{p}) R_{-1}D.
  \]
  In $H^\bullet(M, \mathbb{Q})$, we have that $\ch_{-1}(\mathbf{p}) = 0$ for degree reasons and $\ch_0(\mathbf{p}) = -r$ by Lemma \ref{lmm:ch0-ch1-identities}. Hence we get $S_{-1} = - R_{-1}$. Next we show $T_{-1} = 0$. The second sum is empty. In the first sum, to have both $a + d^L - 2$ and $b + d^R - 2$ nonnegative, we must have $a + b - 2 = a + b + d^L - 2 + d^R - 2 \geq 0$, but $a + b = 1$.
  
  For $k = 0$, again we assume that $D = \prod_j \ch_{i_j}(\gamma_j)$ with $\gamma_j$ of pure degree. By induction we see that $R_0D = (\deg D) D$. By using that $\ch_0(\mathbf{p}) = -r$ and $\ch_1(\mathbf{p}) = 0$, we compute that $S_{0}D = -D$. Finally, consider $T_0$. In the second sum we need to consider the K\"unneth decomposition of $\mathbf{p}$, which is $\mathbf{p} \otimes \mathbf{p}$. Also, since $X$ only has $(p,p)$-cohomology, $\chi(X, \mathcal{O}_X) = 1$. Since $a + b = 0$ in this sum, $a = b = 0$ and the second sum becomes $\ch_0(\mathbf{p})\ch_0(\mathbf{p}) = r^2$.

  In the first sum we have that $a + b = 2$. Since $\ch_1(\gamma) = 0$ by Lemma~\ref{lmm:ch0-ch1-identities}, we do not need to consider $a = b = 1$. If $a = 0$ and $b = 2$ we must have $d^L = 2$ and $d^R = 0$, otherwise the factorials become negative. So we only have to deal with the K\"unneth component in $H^4(X, \mathbb{Q}) \otimes H^0(X, \mathbb{Q})$, which is $\mathbf{p} \otimes 1$. So for $a = 0$ and $b = 2$ we get $-\ch_0(\mathbf{p})\ch_2(1) = r\ch_2(1)$. For $a = 2$ and $b = 0$ we get the same result, so taking everything together we find that $T_0 = -2r\ch_2(1) + r^2$. Since $\ch_2(1)$ is of degree zero, we can compute it by picking a point $[E] \in M$ and noticing that $\ch_2(1)|_{[E]} = \ch_2(-E \otimes \det E^{-1/r})$ by the push-pull formula. We have fixed $r$, $\det E$ and $c_2$, so we can calculate this in a similar manner to Lemma \ref{lmm:ch0-ch1-identities} and obtain that
  \[
    2r\ch_2(1) = 2r\ch_2(-E\otimes \det E^{-1/r}) = 2r\frac{2rc_2 - c_1(\Delta)^2(r - 1)}{2r} = \vdim M + (r^2 - 1)\chi(X, \mathcal{O}_X)
  \]
  Keeping in mind that $\chi(X, \mathcal{O}_X) = 1$, we find that $T_0 = -\vdim M - r^2 + 1 + r^2 = -\vdim M + 1$. Finally, $T_0 + S_0 = -\vdim M$. Then we obtain
  \[
    \int_{[M]^{\mathrm{vir}}} \mathcal{L}_0D = (\deg D - \vdim M) \int_{[M]^{\mathrm{vir}}} D.
  \]
  If the integral is nonzero, then $\deg D = \vdim M$, in which case the first factor vanishes.
\end{proof}

\subsection{The Virasoro bracket}

The operators $L^{\GW}_k$ from Gromov-Witten theory satisfy the Virasoro bracket. In our situation there is a Virasoro bracket as well, but it requires new notation. This notation is much more convenient than the notation we employed before, which we only use because of its history, and because the new notation was only discovered after the rest of the paper was already written.

\begin{definition}
  For $i \geq 0$ and $\gamma \in H^\bullet(X, \mathbb{Q})$ of pure degree, define $h_i(\gamma)$ as $i!\ch_{i + 2 - \deg \gamma}(\gamma) \in \mathbb{D}^X$. Extend the definition to all $\gamma$ by linearity. Let $\mathbb{D}^X_+$ be the subalgebra of $\mathbb{D}^X$ generated by the $h_i(\gamma)$.
\end{definition}

Note that $h_i(\gamma)$ always has degree $i$. The next proposition is immediate.

\begin{proposition}
  The subalgebra $\mathbb{D}^X_+$ is the algebra of elements of nonnegative degree.
\end{proposition}

\begin{definition}
  Fix integers $r$ and $k$ with $k \geq -1$. Define the operator $R_k^+$ on $\mathbb{D}_+^X$ as a derivation, which acts as $R_k^+(h_i(\gamma)) = ih_{i+k}(\gamma)$ on generators. For $\gamma_1$ and $\gamma_2$ of pure degree, define
  \[
    t_k(\gamma_1, \gamma_2) = \sum_{a + b = k} (-1)^{2 - \deg \gamma_1}h_a(\gamma_1)h_b(\gamma_2).
  \]
  Extend the definition by bilinearity. Then define the operator $T_k^+$ as multiplication by the constant element
  \[
    T_k^+ = \sum_i t_k(\gamma_i^L, \gamma_i^R)
  \]
  where $\sum_i \gamma_i^L \otimes \gamma_i^R = \Delta_*\td_X$, the K\"unneth decomposition of the Todd class of $X$. Finally let $S_k^+$ be defined by $S_k^+D = \frac{1}{r}R_{-1}(h_{k+1}(\mathbf{p})D)$. Let $L_k^+ = R_k^+ + T_k^+$ and $\mathcal{L}_k^+ = L_k^+ + S_k$.
\end{definition}

The operators $R_k^+$, $T_k^+$ and $S_k^+$ almost agree with their counterparts of Definition \ref{def:virasoro-operator}. In fact for $k \geq 0$ they agree on all elements of $\mathbb{D}^X_+$, see below. For $k = -1$ we have $T_k^+ = T_k = 0$, so they are the same as well. Finally, $R_{-1}^+$ and $R_{-1}$ agree on elements of positive degree, but not on elements of degree zero. Indeed, $R_{-1}$ sends elements of degree zero to elements of degree $-1$, while $R_{-1}^+$ simply sends those to zero.

\begin{proposition}
  For $k \geq 0$,  we have $R_k^+ = R_k$, $T_k^+ = T_k$ and $S_k^+ = S_k$. Furthermore, Conjecture \ref{sec:strong-conjecture} holds if and only if for every $k \geq -1$, $r \geq 1$, $\Delta$ a line bundle on $X$, $c_2$ an integer and $H$ a polarisation on $X$ such that $M=M_X^H(r, \Delta, c_2)$ contains only stable sheaves, and $D$ any element of $\mathbb{D}_+^ X$, we have
  \begin{equation}\label{eq:conjecture-alternative}
    \int_{[M]^{\text{vir}}} \mathcal{L}^+_kD = 0.
  \end{equation}
\end{proposition}
\begin{proof}
  For $R_k = R_k^+$, one just has to verify it for the $h_i(\gamma)$, which is immediate. We noted before that $R_{-1}^+ = R_{-1}$ for elements of positive degree, and since $h_{k+1}(\mathbf{p})D$ has positive degree for $k \geq 0$ and $D \in \mathbb{D}^X_+$, $S_k = S_k^+$ for $k\geq0$. The equality $T_k  = T_k^+$ means that these elements simply coincide. Recall that $\td_X = 1 + \frac{c_1(X)}{2} + \frac{c_1(X)^2 + c_2(X)}{12}$. The K\"unneth decomposition of $\Delta_*\frac{c_1(X)}{2}$ is $\frac{c_1(X)}{2} \otimes \mathbf{p} + \mathbf{p} \otimes \frac{c_1(X)}{2}$, but $t_k\left(\mathbf{p}, \frac{c_1(X)}{2}\right) = -t_k\left(\frac{c_1(X)}{2}, \mathbf{p}\right)$ because of the sign in the definition, so this part does not contribute. The K\"unneth decompositions of $\Delta_*1$ and $\Delta_*\frac{c_1(X)^2 + c_2(X)}{12}$ correspond to the two sums in Definition \ref{def:virasoro-operator}. The only thing there is to check is that the signs in both definitions are the same, which is not difficult.

  The second claim follows because we can use Lemma \ref{lmm:low-term-lemma} to see that the conjecture is automatic if $D \notin \mathbb{D}^X_+$. Thus the second claim for $k \geq 0$ follows immediately. For $k = -1$ the statement of Conjecture~\ref{sec:strong-conjecture} is equivalent to the statement of the proposition because they are both true: for the conjecture it follows from Proposition~\ref{prp:k-0--1} and for the proposition it follows because $\mathcal{L}_{-1}^+ = 0$.
\end{proof}

Another advantage of using the new notation, is that $L_k^+ = R_k^+ + T_k^+$ satisfies the Virasoro bracket in full generality. This is not true if we consider only $R_k + T_k$ on $\mathbb{D}^X$. It is also not true in the stable pair setting \cite{MOOP_unpublished}, where the bracket is only satisfied after introducing a new formal symbol and using a weaker notion of equality of operators. Finally, note that $L_k^+$ does not depend on the rank $r$, only $\mathcal{L}_k^+$ does so.

\begin{proposition}
  \label{prp:virasoro-bracket}
  The operator $L^+_k$ satisfies the Virasoro bracket as operators on $\mathbb{D}_+^X$, i.e.
  \begin{equation} \label{eq:virasoro-bracket-sheaves} [L^+_k, L^+_m] = (m - k)L^+_{m + k} \end{equation}
  for all $m, k \geq -1$. 
\end{proposition}

For the proof, we first prove two lemma'.

\begin{lemma}
  \label{lmm:r-k-bracket} 
  The operators $R^+_k$ satisfy the Virasoro bracket: $[R_k^+, R_m^+] = (m - k)R_{m+k}^+$ for all $k, m \geq -1$.
\end{lemma}
\begin{proof}
  The commutator $[R^+_k, R^+_n]$ is again a derivation and we have
  \[
    R^+_kR^+_nh_i(\gamma) - R^+_nR^+_kh_i(\gamma) = i(i+n)h_{i + n +k}(\gamma) - i(i + k)h_{i+n+k}(\gamma) = (n - k)R^+_{n+k}h_{i}(\gamma).
  \]
  So they also agree on generators.
\end{proof}

\begin{lemma}
  For all $m, k \geq -1$, equation \eqref{eq:virasoro-bracket-sheaves} is equivalent to
  \begin{equation}\label{eq:bracket-alternative}
    R^+_k(T^+_m) - R^+_m(T^+_k) = (m - k)T^+_{m+k}
  \end{equation}
  
\end{lemma}

\begin{proof}
  Expanding $L^+_k = R^+_k + T^+_k$ in equation \eqref{eq:virasoro-bracket-sheaves} gives the equation
  \[
    [R^+_k, R^+_m] + [R^+_k, T^+_m] + [T^+_k, R^+_m] + [T^+_k, T^+_m] = (m - k)R^+_{m+k} + (m - k)T^+_{m+k}.
  \]
  By Lemma \ref{lmm:r-k-bracket} and by noting that $[T^+_k, T^+_m] = 0$, we get
  \begin{equation}
    \label{eq:bracket-alternative-2}
    [R^+_k, T^+_m] + [T^+_k, R^+_m] = (m - k)T^+_{m+k}
  \end{equation}
  Finally, note that $[R^+_k, T^+_m] = R^+_k(T^+_m)$ since $R^+_k$ is a derivation and $T^+_m$ is a constant, as
  \[
    R^+_k(T^+_mD) - T^+_mR^+_k(D) = R^+_k(T^+_m)D + T^+_mR^+_k(D) - T^+_mR^+_k(D) = R^+_k(T^+_m)D.
  \]
  holds for all $D$.
\end{proof}

\begin{proof}[Proof of Proposition \ref{prp:virasoro-bracket}]
  We will show that the following analog of \eqref{eq:bracket-alternative} holds for $\gamma_1$ and $\gamma_2$ of pure degree:
  \[
    R_k^+(t_m(\gamma_1, \gamma_2)) - R^+_m(t_k(\gamma_1, \gamma_2)) = (m - k)t_{m+k}(\gamma_1, \gamma_2).
  \]
  Given the above expression of $T_k^+$, this immediately implies equation \eqref{eq:bracket-alternative} and hence completes the proof of the proposition. Assume for simplicity that $(-1)^{2 - \deg \gamma} = 1$, this does not affect the proof in an essential way. Note that both sides of the equation are a linear combination of $h_a(\gamma_1)h_b(\gamma_2)$ with $a + b = m + k$. We count how often each $h_a(\gamma_1)h_b(\gamma_2)$ occurs on the left hand side.

  We can get terms $h_a(\gamma_1)h_b(\gamma_2)$ by either applying $R_k$ to $h_{a - k}(\gamma_1)h_b(\gamma_2)$ or $h_{a}(\gamma_1)h_{b - k}(\gamma_2)$ or by applying $R_m$ to $h_{a - m}(\gamma_1)h_b(\gamma_2)$ or $h_a(\gamma_1)h_{b - m}(\gamma_2)$. But if $b - m < 0$, for example, we get zero automatically for $R_m(h_a(\gamma_1)h_{b - m}(\gamma_2))$. Therefore it is useful to distinguish whether $a$ is less than, equal to, or more than $k$, and similarly for $b$. For example, if $a > k$ and $b > k$, then we also have $a < m$ and $b < m$. So in this case only $R_k(h_{a-k}(\gamma_1)h_b(\gamma_2))$ and $R_k(h_{a}(\gamma_1)h_{b - k}(\gamma_2))$ will contribute. The first contributes $(a - k)h_a(\gamma_1)h_b(\gamma_2)$ and the second contributes $(b - k)h_a(\gamma_1)h_b(\gamma_2)$. So the total contribution is $a - k + b - k = m + k - 2k = m - k$, which is exactly the same as the coefficient on the right hand side. The other cases are similar.
\end{proof}

\begin{remark}
  There is an alternative approach to equation \eqref{eq:bracket-alternative} (or rather, equation \eqref{eq:bracket-alternative-2}). If one knows this equation for $k = 1$ and all $m$, one can perform an inductive argument to show that if the equation holds for some $k \neq 1$, then it also holds for $k + 1$. Only Lemma \ref{lmm:r-k-bracket} and the Jacobi identity are needed for this argument. Thus it suffices to check $k = 1$, $k = -1$ and $k = 2$ to complete the proof.
\end{remark}

Finally, we list some more bracket relations. First, we have
\[
  [L_n^+, h_k(\mathbf{p})] = kh_{n+k}(\mathbf{p}).
\]
This relation also appears in \cite{MOOP_unpublished}. We also have the relation
\[
  [L_{-1}^+, S_k^+] = (k+1)S_{k-1}^+.
\]
This implies that $[L_{-1}^+, \mathcal{L}_k^+] = (k+1)\mathcal{L}_{k-1}$. In view of the above remark, one might hope that a similar inductive argument might be used to shed light on some parts of Conjecture \ref{sec:strong-conjecture}, but the author has not succeeded in this.

\subsection{Deformation invariance}

Let $S$ be a smooth $\mathbb{C}$-scheme and consider a smooth family of surfaces $\mathcal{X} \to S$. Let $r$ be a number, $\Delta$ a line bundle on $\mathcal{X}$ and $c_2$ be a cohomology class in $H^4(\mathcal{X}, \mathbb{Z})$. Then for each $s \in S$, we can construct the moduli space $M_s$ of stable sheaves on $\mathcal{X}_s$ of rank $r$, determinant $\Delta|_{\mathcal{X}_s}$ and second Chern class $c_2|_{\mathcal{X}_s}$.

\begin{proposition}
  \label{prp:deformation-invariance}
  Assume that $\mathcal{X}_s$ has only $(p,p)$-cohomology for each $s \in S$. Then the set of points $s \in S$ such that $M_s$ satisfies the Conjecture \ref{sec:strong-conjecture} is open and closed.
\end{proposition}
\begin{proof}
  The $M_s$ are fibres of the relative moduli space of stable sheaves $\mathcal{M} \to S$, see \cite[Sec. 4.3]{huybrechts10}. A familiy of universal sheaves exists \'etale locally, so again the sheaf $\mathcal{E}^{\otimes r} \otimes \det \mathcal{E}^ {-1}$ exists on the relative moduli space. Hence we can also construct the classes \eqref{eq:kernel-classes}. Let $s$ be a closed point in $S$. By the Ehresmann fibration theorem, analytically locally around $s$, the family $\mathcal{X}$ is diffeomorphic to a trivial family $\mathcal{X}_s \times S \to S$. Notably, the cohomology of the fibres is the same. So on a analytic neighbourhood around $s$, we can consider the $\ch_i(\gamma)$ as a family of cohomology classes on $\mathcal{M}$. There exists a relative perfect obstruction theory on $\mathcal{M}$ over $\mathbb{A}^1$, which restricts to the usual obstruction theory over the fibres. Thus the virtual fundamental class is deformation-invariant \cite[Prop. 7.2]{behrend97_intrin_normal_cone}. Hence an integral over a polynomial in the $\ch_i(\gamma)$ is locally constant around $s$. This implies our result.
\end{proof}

\section{Invariant sheaves}
\label{sec:invariant-sheaves}

We introduce a combinatorial description of equivariant sheaves on toric surfaces. This description was first found by Klyachko \cite{Klyachko_Equiv_bundle_toral_var} and was later elaborated by Perling \cite{Perling_2004}, who also introduced new notation. Kool \cite{KOOL20111700} proved that one can use the theory to describe the fixed point locus of the moduli space of stable sheaves of toric varieties. Or presentation follows \cite{kool09_euler_charac_modul_spaces_torsion}.

\subsection{Generalities on smooth projective toric surfaces}

We briefly recall the basic theory of toric varieties that we need. This material can be found in \cite{fulton93_toric}. Assume that $X$ is a toric surface. Associated to $X$ is a fan $\Delta$ in $N$, a free abelian group of rank two. Let $M = N^\vee$. We then have the natural pairing $\langle -,-\rangle : M \otimes N \to \mathbb{Z}$. Associated to each cone $\sigma \in \Delta$ is the set $S_\sigma = \{ m \in M \mid \langle m, s \rangle \geq 0 $ for all $s \in \sigma\}$. Denote by $\mathbb{C}[S_\sigma]$ the ring generated by formal symbols $z^m$ for $m \in S_\sigma$ with multiplication $z^{m_1} \cdot z^{m_2} = z^{m_1 + m_2}$ and let $U_{\sigma}$ be $\Spec \mathbb{C}[S_\sigma]$. If $\sigma_1 \subseteq \sigma_2$ is an inclusion of cones, then we have a canonical open embedding $U_{\sigma_1} \subseteq U_{\sigma_2}$. If we glue the $U_{\sigma}$ along all possible inclusions, we recover $X$. Recall that the fact that $X$ is proper is equivalent to the union of the cones in $\Delta$ being equal to $N$. If this is the case, note the following: each two-dimensional cone $\sigma$ is bordered by two rays $\rho_1$ and $\rho_2$. Then $X$ is smooth if and only if the primitive generators $v_1$ and $v_2$ of $\rho_1$ resp. $\rho_2$ form a basis of $N$ and this holds for each two-dimensional cone $\sigma$. Finally, recall that any smooth proper surface is projective \cite[Sec. 9.3.1]{liu06_algeb}.

The Chow ring of a smooth toric variety can be computed as follows. Enumerate the rays in $\Delta$ as $\rho_1, \rho_2, \ldots ,\rho_d$ with primitive generators $v_1, v_2, \ldots, v_d$. For each $1 \leq i \leq d$ we have a generator $D_i$. These are subject to the following relations:
\begin{enumerate}
\item For each $m \in M$, add the relation $\sum_{i = 1}^d \langle m, v_i\rangle D_i = 0$.
\item For each subset $A$ of $\{1, \ldots, d\}$ such that the $v_i$ for $i \in A$ do not generate a cone of $\Delta$ add the relation $\prod_{i \in A} D_i = 0$.
\end{enumerate}

For the first type of relation one can restrict to a basis of $M$. In this language, the class of the canonical sheaf $\omega_X$ is $- \sum_{i = 1}^d D_i$. In particular, $-\omega_X$ is effective.

\begin{example}
  \label{ex:p2-tor-var}
  There is a natural action of $\mathbb{G}_m^2$ on $\mathbb{P}^2$ given by $(s, t) \cdot (x:y:z) = (s^{-1}x : t^{-1}y : z)$. This makes $\mathbb{P}^2$ into a smooth toric variety. The associated fan is $M = \mathbb{Z}^2$ with rays generated by $(1,0)$, $(0, 1)$ and $(-1, -1)$. The three two dimensional cones are generated by two of these vectors. Write $\Spec \mathbb{C}\left[\frac{X}{Z}, \frac{Y}{Z}\right]$ etc. for the usual charts, the induced action on these rings is given by the relations $(s, t)\cdot X = sX$, $(s, t) \cdot Y = tY$ and $(s, t) \cdot Z = Z$.
\end{example}

\subsection{Equivariant sheaves}

The trivial cone $\{0\}$ in $\Delta$ corresponds to a two-dimensional torus $T = \Spec \mathbb{C}[U_{\{0\}}] = \Spec \mathbb{C}[M]$. Then $T$ acts on itself via left multiplication and this action can always be uniquely extended to $X$. The $U_\sigma$ are preserved under the action. We are interested in coherent sheaves which are equivariant under this action. We recall the definition.

\begin{definition}
  Let $G$ be a group scheme and $X$ a $G$-scheme (e.g. $X$ is toric and $G$ is the corresponding torus). Denote by $\mu : G \times G \to G$ the multiplication and by $\sigma: G \times X \to X$ the action. An \emph{equivariant structure} on a sheaf $F$ is an isomorphism
  \[
    \phi : \sigma^* F \to \pi_X^* F
  \]
  of sheaves on $G \times X$ such that the cocycle condition holds on $G\times G\times X$: $\pi_{23}^*\phi \circ (\id_G \times \sigma)^* \phi = (\mu \times \id_X)^*\phi$, where $\pi_{23}$ is the projection to the second and third factor.
\end{definition}

If $X$ and $G$ are affine, then an equivariant sheaf $F$ is just a module over $\Gamma(X, \mathcal{O}_X)$ together with a coaction of the coalgebra $\Gamma(G, \mathcal{O}_G)$. In particular, if $X$ is a toric variety and $G = T$, the corresponding torus, we can describe an equivariant sheaf by specifying for each $\sigma$ a module over $\Gamma(U_{\sigma}, \mathcal{O}_X)$ with a $\Gamma(T, \mathcal{O}_T)$-action, such that for $\sigma_1$ and $\sigma_2$, the modules and their actions agree on the overlap $U_{\sigma_1 \cap \sigma_2}$. Clearly, if a sheaf has an equivariant structure, then it has several, by multiplying with a character of $T$. 

We can exploit the affine cover $U_{\sigma}$ of $X$ to find a combinatorial description of equivariant sheaves. We will use this description to find all equivariant sheaves with certain numerical invariants. We introduce the following combinatorial data due to Perling \cite{Perling_2004}:

\begin{definition}
  Let $\sigma \subseteq \Delta$ be a maximal cone. A \emph{$\sigma$-family} is an $M$-graded vector space $\{F_m\}_{m \in M}$ together with a morphism $\chi_s : F_\bullet \to F_{\bullet + s}$ for each $s \in S_\sigma$ such that $\chi_{s_1 + s_2} = \chi_{s_1} \circ \chi_{s_2}$ and $\chi_0 = \id$.\\
  A $\sigma$-family is called finite if there are a finite number of homogeneous generators.
\end{definition}

An equivariant sheaf $F$ on an affine toric variety $U_\sigma$ gives rise to a $\sigma$-family $\hat{F}$ as follows. First we identify $F$ with the $\mathbb{C}[S_\sigma]$-module $H^0(U_\sigma, F)$. Then the affine group $T = \Spec \mathbb{C}[M]$ acts on $F$, and since every action of a torus is diagonisable, $F$ decomposes into weight spaces as $F = \bigoplus_{m \in M} \hat{F}_m$. Multiplication by $z^s \in \mathbb{C}[S_{\sigma}]$ induces a map $\hat{F}_m \to \hat{F}_{m + s}$. It is not difficult to see that this assignment extends to an equivalence of categories between equivariant coherent sheaves on $U_\sigma$ and finite $\sigma$-families. In the following, the notation $\hat{F}$ will always mean the $\sigma$-family associated to a coherent sheaf $F$.

Let $\sigma$ be a two-dimensional cone. For the smooth affine $U_\sigma$, there is a more concrete description of $S_\sigma$, and hence of the $\sigma$-families. Let $v_1$ and $v_2$ generate the two boundary rays of $\sigma$, then by smoothness this is a basis of $N$. Hence we obtain a dual basis $w_1$, $w_2$ for $M$, and $S_\sigma$ is exactly the set of positive linear combinations of the $w_i$. This implies that $\mathbb{C}[S_\sigma] \cong \mathbb{C}[z^{w_1}, z^{w_2}]$, the usual polynomial ring in two variables. Let $\hat{F}$ be a $\sigma$-family. Define $\hat{F}(n_1, n_2)$ as $\hat{F}_{n_1w_1 + n_2w_2}$ for integers $n_1, n_2$. Note we have maps
\begin{equation}\label{eq:sigma-maps}
  \hat{F}(n_1, n_2) \to \hat{F}(n_1 + 1, n_2) \quad\text{and}\quad \hat{F}(n_1, n_2) \to \hat{F}(n_1, n_2 + 1)
\end{equation}
by multiplication with $z^{w_1}$ and $z^{w_2}$, respectively. The $\hat{F}(n_1, n_2)$ together with these two maps completely determine $\hat{F}$. In fact, we obtain again an equivalence of categories between $\sigma$-families $\hat{F}$ and families $\hat{F}(n_1, n_2)$ with maps as in \eqref{eq:sigma-maps} which make all the squares commute. It is convenient to picture a lattice $\mathbb{Z}^2$ with $\hat{F}(n_1, n_2)$ sitting at the point $(n_1, n_2)$, with horizontal maps going to the right and vertical maps going upwards.

We will now explain how these $\sigma$-families glue to produce equivariant coherent sheaves on $X$. This is easier to describe in the case of a torsion-free sheaf. Since this is the only case we will need, we assume that all our equivariant sheaves are torsion-free from now on. We have the following characterisation in terms of $\sigma$-families.

\begin{lemma}[{\cite[Prop. 5.13]{Perling_2004}}]
  Let $F$ be an equivariant coherent sheaf on $U_\sigma$. Then $F$ is torsion-free if and only if all $\chi_s$ are injective, if and only if all the maps of \eqref{eq:sigma-maps} are injective.\\
  As a consequence, the category of torsion-free equivariant sheaves on $U_\sigma$ is equivalent to the category of finite $\sigma$-families with the additional assumption that the maps $\hat{F}_m \to \hat{F}_{m + s}$ are \emph{inclusions}.
\end{lemma}

The $\sigma$-family $\hat{F}$ of an equivariant coherent sheaf $F$ on $U_\sigma$ is finitely generated. Hence, for large $n_1$, $n_2$, the inclusions of \eqref{eq:sigma-maps} are actually identities. Assume without loss of generality that this limiting space is $\mathbb{C}^r$. Then $r$ is the rank of $F$. Furthermore, for each fixed $n_1$, we can produce a filtration of $\mathbb{C}^n$. Indeed, since the $\sigma$-family is finite the space $\hat{F}(n_1, n_2)$ is constant for sufficiently large $n_2$. Denote this space by $\hat{F}(n_1, \infty)$. Varying $n_1$ gives us sequence of inclusions

\begin{equation}
  \label{eq:flag}
  \ldots \subseteq \hat{F}(n - 1, \infty) \subseteq \hat{F}(n, \infty), \subseteq \hat{F}+1,\infty) \subseteq \ldots
\end{equation}

This sequence is easily seen to be a finite full flag, defined below.

\begin{definition}
  A \emph{finite full flag} of a vector space $V$ is a sequence of vector spaces $V_\lambda$, with $\lambda \in \mathbb{Z}$ such that $V_{\lambda} \subseteq V_{\lambda+1}$, and $V_\lambda = V$ for sufficiently large $\lambda$ and $V_\lambda = 0$ for sufficiently small $\lambda$. We call $V_{\lambda}$ the space of \emph{weight $\lambda$} of the flag.
\end{definition}

The next theorem is Klyachko's description of equivariant torsion-free sheaves. See also \cite[Sec. 5.4]{Perling_2004} for this and a more general theorem. 

\begin{theorem}
  Let $X$ be a smooth projective toric surface with fan $\Delta$. Suppose that we have a finite $\sigma$-family $\hat{F}_\sigma$ for each two-dimensional $\sigma \in \Delta$ where all maps are inclusions and with limiting space $\mathbb{C}^r$. Then the $F_\sigma$ glue to an equivariant coherent sheaf on $X$ if and only if  condition $(\star)$ holds for each two-dimensional $\sigma_1$ and $\sigma_2$ that share a boundary ray with primitive generator $v$.

  \begin{itemize}
  \item[\emph{($\star$)}]\label{glue-cond} 
    Let $(v_1, w_1)$ and $(v_2, w_2)$ be the (ordered) bases of $M$ associated to $\sigma_1$ and $\sigma_2$ as above, where $v_1(v) = v_2(v) = 1$ and $w_1(v) = w_2(v) = 0$. Then the full finite flags of $\mathbb{C}^r$ given by $F_{\sigma_1}(n, \infty)$ and $F_{\sigma_2}(n, \infty)$ are equal.
  \end{itemize}
  
  Furthermore, this gives an equivalence of categories between equivariant torsion-free coherent sheaves on $X$ and collections of finite $\sigma$-families $\hat{F}_\sigma$ with all $\chi_s$ injective satisfying $(\star)$.
\end{theorem}

Now, we can translate properties of ordinary sheaves into the language of these $\sigma$-families. Recall that a sheaf $F$ is \emph{reflexive} if the canonical map $F \to F^{\vee\vee}$ of $F$ into its double dual is an isomorphism. Every reflexive sheaf is torsion-free. On the other hand, using results from \cite[][Sec. 1.1]{huybrechts10}, if $F$ is torsion-free then the map $F \to F^{\vee\vee}$ is injective and the quotient is zero-dimensional. Furthermore, it follows from \cite[][Sec. 1.1]{huybrechts10} that a coherent sheaf on a surface is reflexive if and only if it is locally free. Therefore, we state the following characterisation of equivariant reflexive sheaves on surfaces, which generalises to higher-dimensional smooth toric varieties.

\begin{proposition}[{\cite[Sec. 5.5]{Perling_2004}}]\label{prp:reflexive-sheaves}
  Let $X$ be a smooth projective toric surface and let $F$ be a torsion-free equivariant sheaf on $X$. Then $F$ is reflexive if and only if for each two-dimensional cone $\sigma$ we have $\hat{F}_\sigma(n_1, n_2) = \hat{F}_\sigma(n_1, \infty) \cap \hat{F}_{\sigma}(\infty, n_2)$, where $\hat{F}_\sigma$ is the $\sigma$-family on $U_{\sigma}$ associated to $F$.
\end{proposition}

In general, we only have the inclusion $\hat{F}_{\sigma}(n_1, n_2) \subseteq \hat{F}_\sigma(n_1, \infty) \cap \hat{F}_{\sigma}(\infty, n_2)$. Since we are dealing with finite $\sigma$-families, we also know that equality fails only in a finite number of cases. This gives an easy description of $F^{\vee\vee}$: we simply define $\hat{F}^{\vee\vee}(n_1, n_2)$ to be $\hat{F}(n_1, \infty) \cap \hat{F}(\infty, n_2)$. Then indeed $F^{\vee\vee}$ is reflexive by Prop.~\ref{prp:reflexive-sheaves} and there is a natural inclusion $F \to F^{\vee\vee}$ with a cokernel that is finite-dimensional as a vector space (implying that it is supported in dimension zero (see \cite[Prop. 2.8]{KOOL20111700})).

As a corollary of the proposition, an equivariant vector bundle on $X$ of rank $n$ is completely determined by a finite complete flag for each ray of the fan $\Delta$ associated to $X$. An equivariant line bundle has an even easier description, as a finite complete flag of $\mathbb{C}$ can be described by giving a number $m$: the flag is then given by $V_\lambda = 0$ if $\lambda < m$ and $V_\lambda = \mathbb{C}$ if $\lambda \geq m$. Thus to describe an equivariant line bundle, one needs to give an integer for each ray.

\begin{example}
  \label{ex:p2-tan-bundle}
  The tangent bundle on $\mathbb{P}^2$ has a canonical equivariant structure. On the affine chart $\Spec \mathbb{C}[\frac{X}{Z}, \frac{Y}{Z}]$, the tangent sheaf is generated by $\frac{\partial}{\partial X/Z}$ and $\frac{\partial}{\partial Y/Z}$. The action of $(s, t)$ on these generators is given by $s^{-1}$ and $t^{-1}$ respectively. Hence the $\sigma$-family on this cart can be pictured as follows. We have the $\mathbb{Z}^2$-grid. The point with coordinates $(m, n)$ corresponds to the eigenspace of $s^mt^n$. The vectors $\frac{\partial}{\partial X/Z}$ and $\frac{\partial}{\partial Y/Z}$ generate the weight spaces of weight $s^{-1}$ and $t^{-1}$ respectively. There are no nontrivial relations so if $m, n$ are nonnegative we get the space $\mathbb{C}\frac{\partial}{\partial X/Z} \oplus \mathbb{C} \frac{\partial}{\partial Y/Z} \cong \mathbb{C}^2$ and on the nodes $(-1, n)$ and $(m, -1)$ with $m$ and $n$ nonnegative we get $\mathbb{C}\frac{\partial}{\partial X/Z}$ and $\mathbb{C}\frac{\partial}{\partial Y/Z}$ respectively. One of the limiting flags is $\ldots \subseteq 0 \subseteq \mathbb{C}\frac{\partial}{\partial X/Z} \subseteq \mathbb{C}^2 \subseteq \ldots$. One can similarly calculate the other flag and also do this for the other charts. The result is always the same: a flag $\ldots \subseteq 0 \subseteq \mathbb{C} \subseteq \mathbb{C}^2 \subseteq \ldots$ where $\mathbb{C}$ has weight $-1$.
  
  We view the entire tangent bundle as a picture as a triangle with three ``strips'', see Figure \ref{fig:tang-p2}. The corners of the triangle represent the three charts of $\mathbb{P}^2$. The center triangle corresponds to the part where the weight spaces have dimension 2. The strips correspond to the part where they have dimension 1. The case described above looked exactly like such a corner: there is an area, bounded on two sides, where the weight space has dimension 2. Furthermore, there were two strips where they have dimension 1. One can also see that they satisfy condition $(\star)$, which is represented by the dotted lines.

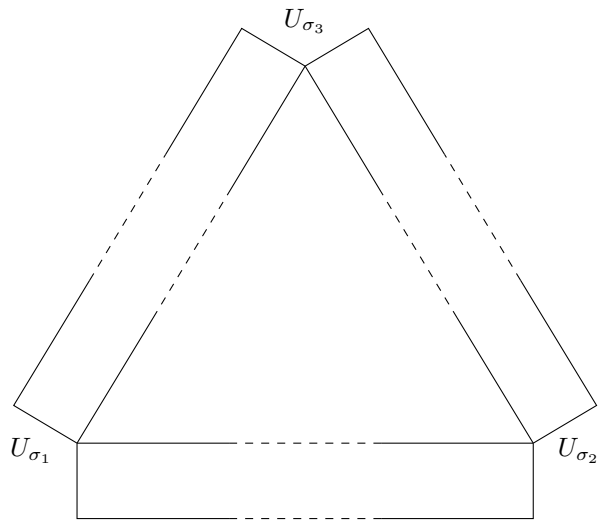
\begin{figure}
  \centering
  \begin{tikzpicture}
    \draw (0,0) -- (2,0);
    \draw [dashed] (2,0) -- (4,0);
    \draw (4,0) -- (6,0);
    \draw (0,0) -- (1,1.66);
    \draw [dashed] (1,1.66) -- (2,3.33);
    \draw (2,3.33) -- (3, 5);
    \draw (6,0) -- (5,1.66);
    \draw [dashed] (5,1.66) -- (4,3.33);
    \draw (4,3.33) -- (3,5);

    \draw (0,0) -- (0,-1) -- (2,-1);
    \draw [dashed] (2,-1) -- (4,-1);
    \draw (4,-1) -- (6,-1) -- (6,0);

    \draw (0,0) -- (-0.833, 0.5) -- (0.166, 2.16);
    \draw [dashed] (0.166, 2.16) -- (1.166, 3.83);
    \draw (1.166, 3.83) -- (2.166, 5.5) -- (3,5);

    \draw (6,0) -- (6.833, 0.5) -- (5.833, 2.16);
    \draw [dashed] (5.833, 2.16) -- (4.833, 3.83);
    \draw (4.833, 3.83) -- (3.833, 5.5) -- (3,5);

    \node[xshift = -6mm, yshift = -1mm] at (0,0) {$U_{\sigma_1}$};
    \node[xshift = 6mm, yshift = -1mm] at (6,0) {$U_{\sigma_2}$};
    \node[yshift = 6mm] at (3,5) {$U_{\sigma_3}$};
  \end{tikzpicture}
  \caption{The tangent bundle on $\mathbb{P}^2$.} \label{fig:tang-p2}
\end{figure}


\end{example}

\subsection{Chern classes of equivariant sheaves}

Using the simple description of equivariant line bundles above it is easy to describe their classes in the Chow ring. This will allow us to compute the Chern characters of arbitrary equivariant torsion-free sheaves.

\begin{lemma}
  Let $X$ be a smooth projective toric variety. Suppose there are $d$ rays on $X$, with associated divisor $D_i \in A^1(X)$, for $1 \leq i \leq d$. Let the equivariant line bundle be given by the integer $m_i$ on the $i$-th ray, as above. Then $c_1(L) = - \sum m_iD_i$. 
\end{lemma}

This lemma will allow us to compute all Chern characters of all equivariant sheaves encountered in this paper, as it is generally very easy to find a resolution of an equivariant sheaf. The pictorial representation of the $\sigma$-families is very helpful here. Below we only work out a specific example but it is easy to see how to construct more general resolutions in a similar manner. The abstract theory also ensures resolutions always exist, see \cite{ChrissGinzburg_rep}. Sometimes it is also possible to construct such resolutions explicitely, see \cite[Thm. 6.1]{Perling_2004}.

\begin{example}
  Consider again the tangent bundle on $\mathbb{P}^2$ with it charts $U_{\sigma_1}$, $U_{\sigma_2}$ and $U_{\sigma_3}$. With Figure \ref{fig:tang-p2} in mind, it is easy to construct a resolution. We consider $L_1$, $L_2$ and $L_3$ which are line bundles which are each defined as the center triangle in Figure \ref{fig:tang-p2} plus one of the strips. The map $L_1 \oplus L_2 \oplus L_3$ has a kernel which is a line bundle whose picture is just $\mathbb{C}$ at each point of the center triangle. The resulting resolution is the familiar Euler sequence. The new information we are getting is that it can be upgraded to an equivariant sequence (and that one has to take different equivariant structures on each of the $\mathcal{O}(1)$'s).  
\end{example}

\subsection{Stability of equivariant bundles}

Let $X$ be a smooth toric variety with fixed polarisation $H$. We will give a criterion for an equivariant vector $E$ to be $\mu$-stable on $X$. Note that in general, one is interested in Gieseker semi-stable sheaves, since this notion of stability gives the correct moduli space. We will only apply this criterion in situations where $\mu$-stability and Gieseker stability coincide, so that the criterion below is actually a criterion for Gieseker stability. We first introduce some notation.

Fix a toric variety $X$ with polarisation $H$ and let $E$ be an equivariant vector bundle of rank $r$ over $X$.  For each ray of $\Delta$, the fan of $X$, we have a finite full flag \eqref{eq:flag}. Recall that $d$ is the number of rays of $\Delta$ and let us denote for each $1 \leq i \leq d$ the flag associated to the $i$-th ray by $\hat{E}^i(\lambda)$. We let $\delta^i_j$ be the number of spaces of dimension $j$ in the flag $\hat{E}^i$ (which is always a finite number). If $W \subseteq \mathbb{C}^r$ is any subspace, denote by $w^i_j$ the dimension of the intersection of $W$ with some $\hat{E}^i(\lambda)$ of dimension $j$. This does not depend on the choice of $\lambda$, and if no such $\lambda$ exists, the number $w^i_j$ will not be important and can have arbitrary value. Finally, to each ray $i$ corresponds a divisor $D_i$. We define $\deg D_i$ as $H.D_i$.

The next criterion is established in the course of the proof of Theorem 3.20 (for $\mu$-semi-stability) and in Prop. 4.13 (for $\mu$-stability) in \cite{KOOL20111700}.

\begin{proposition}
  \label{prp:stability-ineq}
  An equivariant vector bundle $E$ of rank $r$ is $\mu$-stable if and only if for each nontrivial subspace $0 \subsetneq W \subsetneq \mathbb{C}^r$ the following inequality holds:
  \[
    \frac{1}{\dim W} \sum_{i = 1}^d \sum_{j = 1}^{r - 1} \delta^i_j \cdot \deg D_i \cdot w^i_j < \frac{1}{n} \sum_{i = 1}^d \sum_{j = 1}^{r - 1} \delta^i_j\cdot \deg D_i\cdot j
  \]
  For $\mu$-semi-stability, replace $<$ by $\leq$ in the above inequality. 
\end{proposition}

\begin{remark}
  If $X$ is a surface, the criterion holds for arbitrary equivariant torsion-free sheaves, instead of just vector bundles. This can be seen as follows. A sheaf $E$ is $\mu$-stable if and only if $E^{\vee\vee}$ is. The latter sheaf is locally free, so the above criterion applies. But the criterion only depends on the limiting flags of $E^{\vee\vee}$, which are the same as the limiting flags of $E$.
\end{remark}

\begin{example}
  The tangent bundle on projective space $\mathbb{P}^2$ is stable. For a sheaf of rank 2 we only have to deal with $\delta^1 = \delta^1_1$, $\delta^2 = \delta^2_1$ and $\delta^3 = \delta^3_1$. We see from the example above that in this case all three numbers are equal to one. The stability inequalities from Proposition \ref{prp:stability-ineq} translate to the triangle inequalities: $\delta_1 < \delta_2 + \delta_3$, $\delta_2 < \delta_3 + \delta_1$ and $\delta_3 < \delta_1 + \delta_2$. These are satisfied, so the bundle is stable.
\end{example}

\subsection{Equivariant K-theory}

In this paragraph we describe a localistation formula for K-theory. Let $X$ be a smooth projective variety on which a torus $T$ acts. Let $X^T$ be the fixed point locus with inclusion $\iota : X^T \to X$. Then $X^T$ is smooth as well. Let $N$ be the normal bundle of $X^T$ in $X$. Denote by $K_T^0(X)$ and $K_T^0(X^T)$ their equivariant K-theories. Then there is a localisation theorem, which was first proven in \cite{10.1215/S0012-7094-92-06817-7}. The following formulation can be found in \cite[Sec. 2.3]{okounkov_k_theory}.

\begin{theorem}
  \label{thm:k-theoretic-localisation}
  There exists finitely many characters $\mu_i$ of $T$ such that the pushforward map $\iota_* : K_T^0(X^T) \to K_T^0(X)$ becomes an isomorphism after localising at $1 - \mu_i$ for all $i$. Furthermore, in this case the class $\bigwedge^\bullet N^\vee \in K^0_T(X^T)$ becomes invertible and we have that
  \[
    F = \iota_*\left(\frac{\iota^*F}{\bigwedge^\bullet N^\vee}\right)
  \]
  for all $F$ in the localised equivariant K-theory of $X$.
\end{theorem}

If $X$ is a toric variety, then $X^T$ is a disjoint union of reduced points, one for each two-dimensional cone in the fan of $X$. The K-theory of a point is the ring of representations. When $X$ is two-dimensional and $T = \mathbb{G}_m^ 2$, this ring is $\mathbb{Z}[s, t]$. Hence $K^0_T(X^T) = \prod_{p \in X^T}\mathbb{Z}[s,t]$.

In this case, Theorem \ref{thm:k-theoretic-localisation} takes the following form. Since $X^T$ is zero-dimensional, $N = T_X|_{X^T}$. For each point $p \in X^T$, we have that $\bigwedge^\bullet N^\vee|_{\{p\}} = \bigwedge^\bullet \Omega_X|_{\{p\}}$. If we write $\Omega_X|_{\{p\}} = \chi_{p, 1} + \chi_{p,2}$ in the representation ring, then $\bigwedge^\bullet \Omega_X|_{\{p\}} = (1 - \chi_{p,1})(1 - \chi_{p,2})$.

We will be interested in computing the Euler characteristic of a sheaf using localisation. For this, we push the equality of Theorem \ref{thm:k-theoretic-localisation} to a point, which gives us:

\begin{equation}
  \label{eq:euler-char-loc}
  \chi(X, F) = \sum_{p \in X^T} \frac{F|_{\{p\}}}{(1 - \chi_{p,1})(1 - \chi_{p,2})}
\end{equation}

\subsection{Equivariant cohomology}

Here we recall Atiyah-Bott localisation, which we use to evaluate integrals on the moduli space $M$. The theory is similar to the equivariant K-theory described above. Let $X$ be a smooth projective variety on which a torus $T$ acts. Again, this implies that the fixed point locus $X^T$ is smooth. Let $\iota : X^ T \to X$ be the inclusion and $N$ be the normal bundle. We have equivariant cohomology groups $H^\bullet_T(X, \mathbb{Q})$ and $H^\bullet_T(X^T, \mathbb{Q})$. Then we have the following theorem:

\begin{theorem}[{\cite[Thm. 2]{Edidin1995LocalizationIE}}]
  \label{thm:a-b-localisation}
  For any equivariant cohomology class $\alpha \in H^\bullet_T(X, \mathbb{Q})$, we have the equality
  \[
    \int_X \alpha = \int_{X^T} \frac{\iota^* \alpha}{e(N)}.
  \]
  Here $e(N)$ denotes the Euler class of the normal sheaf $N$.
\end{theorem}

In the cases where we apply this theorem, $X^T$ is isolated. Then we have that $N = T_X|_{X^T}$. Then the integral to the right becomes a finite sum, which is easier to evaluate.

This theorem can also be used to compute non-equivariant integrals. For this, we use that there is a forgetful map $H^\bullet_T(X, \mathbb{Q}) \to H^\bullet(X, \mathbb{Q})$. To compute the integral of a cohomology class $\alpha$ in $H^\bullet(X, \mathbb{Q})$, if we have a lift $\alpha' \in H^\bullet_T(X, \mathbb{Q})$ then we can use Theorem~\ref{thm:a-b-localisation} to evaluate the integral $\int_X \alpha' \in H^\bullet_T(\{*\}, \mathbb{Q})$. One can then apply the forgetful map to obtain the integral of $\alpha$. 

\section{Verification of the conjecture in special cases}
\label{sec:results-toric-vari}

Here we describe how to check Conjecture \ref{sec:strong-conjecture} on a specific smooth projective toric surface $X$ with a polarisation $H$ in special cases. Fix a rank $r$, and a Chern class $c = 1 + c_1 + c_2$. For $X$ toric, this determines the determinant. Let $M = M^H_X(r, c_1, c_2)$ be the moduli space, as before. We assume that $\gcd(r, \ch_1\cdot H) = 1$. This ensures that $M$ is fine and that Gieseker semi-stability, Gieseker stability and $\mu$-stability all coincide. 

\begin{lemma}
  \label{lmm:smoothness}
  With these assumptions, $\Ext^2(E, E) = 0$ for any stable sheaf. Hence $M$ is smooth of the expected dimension.
\end{lemma}

\begin{proof}
  By Serre duality, $\Ext^2(E, E) = \Hom(E, E\otimes \omega_X)^\vee$. Since $\omega_X$ is anti-effective for $X$ toric, $\deg(\omega_X) = H.\omega_X < 0$. Then $\mu(\omega_X) = \deg(\omega_X) < 0$, so $\mu(E \otimes \omega_X) = \mu(E) + \mu(\omega_X) < \mu(E)$. Also, $E\otimes \omega_X$ is still $\mu$-stable. By Schur's lemma for $\mu$-stable sheaves, the only morphism $E \to E \otimes \omega_X$ is the zero morphism.
\end{proof}

Denote the torus of $X$ by $T$. The action of $T$ lifts to $M$, see \cite{KOOL20111700}. On the level of points, the action of $t \in T$ sends $[E] \in M$ to $[\lambda_t^*E]$, where $\lambda_t: X \to X$ is the multiplication by $t$. In order to evaluate the integral of the conjecture, we use localisation, Theorem~\ref{thm:a-b-localisation}. Therefore, we need to know the fixed point locus $M^T$. This locus consists of the stable sheaves in $M$ admitting an equivariant structure, see \cite{KOOL20111700}. In the cases we consider, $M^T$ is isolated, so the formula of Thm.~\ref{thm:a-b-localisation} becomes much easier. For now we assume that we have an explicit description of all the sheaves in $M^T$, later we will address the problem of finding all such sheaves, which is a difficult problem in general.

The (twisted) universal sheaf $\mathcal{E}$ on $M$ becomes equivariant for the action of $T$ described above. This ensures that the classes \eqref{eq:kernel-classes} admit lifts in equivariant cohomology. Also, for any toric variety $X$, $H^\bullet(X, \mathbb{Q}) \cong A^\bullet(X, \mathbb{Q})$ by \cite[Sec. 5.2]{fulton93_toric} and the latter is generated by classes of invariant subschemes (see the description in Sec. \ref{sec:invariant-sheaves}). Therefore, all cohomology classes of $X$ also admit equivariant lifts. We conclude that all the classes $\ch_i(\gamma)$ admit equivariant lifts. Thus we can indeed apply Theorem~\ref{thm:a-b-localisation} (see the remarks after the Theorem).

\subsection{The computation}

We will explain how to evaluate the equivariant integral $\int_M P$, where $P$ is any polynomial in the $\ch_i(\gamma)$, for $\gamma$ an equivariant cohomology class on $X$. One can set $P = \mathcal{L}_kD$ to verify the conjecture. In particular, we will see that the only input we need is $T_{X, p}$ for $p \in X^T$ as two-dimensional representation of $T$ and $E_{p, q}$, the equivariant K-theory class of the equivariant\footnote{The choice of equivariant structure does not matter for the computation.} sheaf $E_q$ corresponding to $q \in M^T$ at $p \in X^T$. This is a finite amount of discrete data. The author has used the freely available computer algebra system SageMath \cite{sage_9.3} to run the computation.

First we compute $\ch_i(\gamma)$ for any $i$ and $\gamma$ as elements of $H^\bullet_T(M^T, \mathbb{Q})$. Since $X^T \times M^T$ is zero-dimensional, the map $\ch : K_T(X^T \times M^T) \to H_T^\bullet(X^T \times M^T, \mathbb{Q})$ is an isomorphism, so the component of $\ch(\mathcal{E})$ at the point $(p, q) \in X^T \times M^T$ is simply $E_{p, q}$. We can then compute the classes \eqref{eq:kernel-classes} by using a formal power series as mentioned in the introduction. Since $M$ is smooth of the expected dimension $d = 2rc_2 - (r - 1)c_1^2 - (r^2-1)\chi(X, \mathcal{O}_X)$, we can safely ignore the terms in the power series of degree higher than $2d$, so this becomes a finite sum. Now given $\gamma$ with component $\gamma_p$ at $p \in X^T$, localisation allows us to compute:
\[
  \ch_i(\gamma) = \sum_{p \in X^T} \frac{\gamma_p \cdot \ch_i\left(-E_{p,q} \otimes \det(E_{p, q})^{-1/r}\right)}{\text{e}(T_{X, p})}.
\]

This allows us to compute any polynomial in the $\ch_i(\gamma)$, simply by addition and multiplication. If such a polynomial has component $P_q$ at $q \in M^T$, localisation tells us that we want to consider
\[
  \int_M P = \sum_{q \in {M^ T}} \frac{ P_q }{ \text{e}(T_{M, q}) }.
\]
This reduces the problem to computing $T_{M, q}$ as representation of $T$. It is well-known that $T_{M, q} \cong \Ext^1(E, E)_0$, the trace-free Ext-group, and this isomorphism respects the $T$-structure. The codomain of the trace map is $H^1(X, \mathcal{O}_X)$, which is zero for a toric surface. Hence $\Ext^1(E, E)_0 = \Ext^1(E,E)$. Since $\Ext^0(E, E) = 1$ (as representation) and $\Ext^2(E, E) = 0$ by Lemma \ref{lmm:smoothness}, $\Ext^1(E, E) = 1 - \chi(E, E)$ as elements of the representation ring. The Euler class is a K-theoretic invariant, which can be computed by the K-theoretic localisation formula \eqref{eq:euler-char-loc}. This gives us
\[
  \chi(E, E) = \sum_{p \in X^T} \frac{H^0(\{p\}, E_{p, q} \otimes E_{p, q}^\vee)}{(1 - \chi_{p,1})(1 - \chi_{p,2})}.
\]
where $\chi_p^1$ and $\chi_p^2$ are as in \eqref{eq:euler-char-loc}. This completes the calculation. It is evident from the computation that this can be done by a computer, given the data specified before.

\begin{example}
  \label{ex:run-ex-tansp}
  Recall that on $X = \mathbb{F}_0 = \mathbb{P}^1\times \mathbb{P}^1$, we have that $H^2(X, \mathbb{Q})$ is generated by $F$, $Z$, where $F$ is a fibre of the projection $\pi_2 : \mathbb{P}^1 \times \mathbb{P}^1 \to \mathbb{P}^1$ and $Z$ is the image of a section of $\pi_2$. The fan of $X$ has four cones of dimension two, which are each given by a quadrant of $\mathbb{Z}^2$. We number the associated fixed points by letting $X_1$ be the fixed point corresponding to the upper-right quadrant and then proceeding clockwise to define $X_2$, $X_3$ and $X_4$. Then the tangent spaces at these points have representations $s^{-1} + t^{-1}$, $s^{-1} + t$, $s + t$ and $s + t^{-1}$ respectively, where $\mathbb{Z}[s, t]$ is the representation ring of the two-dimensional torus.

  Consider the case where $r = 2$, $\Delta = F+Z$, $c_2 = 2$ and with polarisation $H = 2F + 5Z$. Then the fixed point locus consists of four equivariant sheaves. In the following table we put what the representation of such a sheaf $F$ is when restricted to a fixed point $X_i$.

  \begin{center}
    \begin{tabularx}{.8\textwidth}{
    |>{\centering\arraybackslash}X
    |>{\centering\arraybackslash}X
    |>{\centering\arraybackslash}X
    |>{\centering\arraybackslash}X |}
  \hline $F|_{X_1}$ & $F|_{X_2}$ & $F|_{X_3}$ & $F|_{X_4}$ \\ \noalign{\hrule height 1.5pt}
  $1 + s^{-1}t$& $t^2 + s^{-1}$& $t^2 + 1$& $t + 1$\\ \hline
 $t + t^{-1}$& $t + 1$& $s + t$& $st + t^{-1}$\\ \hline
      $t + 1$& $t^2 + 1$& $t^2 + s$& $st + 1$\\ \hline
      $s^{-1}t + t^{-1}$& $t + s^{-1}$& $t + 1$& $t + t^{-1}$\\ \hline
    \end{tabularx}
  \end{center}

  Of course, one has to make a choice of equivariant structure. We explain a way of doing so when explaining how to find all vector bundles.
  
  As the reader can see, all the elements in the table are honest representations, no minuses occur. This is a sign that all of the sheaves are vector bundles, which is indeed the case. Using K-theoretic localisation \eqref{eq:euler-char-loc}, we can compute that the tangent spaces to $M$ at these sheaves are respectively $t^{-1} + s^{-1} + s^{-1}t^{-1}$, $st + s + t$, $s + st^{-1} + t^{-1}$ and $t + ts^{-1} + s^{-1}$. Their Euler classes are therefore $-s^2t - st^2$, $s^2t + st^2$, $-s^2t + st^2$ and $s^2t - st^2$.   
\end{example}

\subsection{Verifying the conjecture}

Next we explain how to verify the conjecture given the algorithm above and the data $E_{p, q}$ and $T_{X, p}$. Fix a basis $\gamma_k$ of $H^\bullet(X, \mathbb{Q})$, where each $\gamma_k$ is of pure degree. It suffices to verify the conjecture for $D = \prod_j \ch_{i_j}(\gamma_{k_j})$ a monomial, where each $\ch_{i_j}(\gamma_{k_j})$ has positive degree. Indeed, if one of the $\ch_{i_j}(\gamma_{k_j})$ has nonnegative degree we can pull them out of $\mathcal{L}_k$ by Lemma \ref{lmm:low-term-lemma}. Hence we can restrict to this case.

But there are essentially a finite amount of such $D$, since the conjecture is also automatic if $\deg D > \dim M$ by Corollary \ref{sec:large-k}. But one can generate a list of all $D$ with a fixed degree. Hence we should, for each $k$, generate a list of all $D$ with degree $\dim M - k$ and check the conjecture for $\mathcal{L}_kD$.

To perform this check for a given monomial $D$, choose equivariant lifts $\beta_k$ for the $\gamma_k$. Then we perform the computation described above. The result lives in $H^\bullet_T(\{*\}, \mathbb{Q}) = \mathbb{Q}[s, t]$, but we can plug in $s = t = 0$ to obtain a rational number. This is the desired integral $\int_M \mathcal{L}_kD$.

\begin{remark}
  We can in fact restrict to the case where $D$ is a monomial $\prod_j \ch_{i_j}(\gamma_{k_j})$ with $i_j \neq 1$ for all $j$. Indeed, if $i_j = 1$ then $\gamma_{k_j}$ must be a multiple of $\mathbf{p}$, otherwise the degree of $\ch_{i_j}(\gamma_{k_j})$ is not positive. But then we can use Corollary \ref{cor:eliminiating-ch-1} to verify the conjecture.
\end{remark}

\begin{remark}
  Without any modification, this algorithm is not very efficient in terms of running time. It is advisible to fix the $\beta_k$ and to compute $\ch_i(\beta_k)$ in advance, storing the result. Similarly, one might compute the elements $T_k$ in advance. This greatly speeds up the computation when one wants to verify the conjecture for all $D$. 
\end{remark}

\begin{example}
  We continue with Example \ref{ex:run-ex-tansp}. We first choose equivariant lifts of the cohomology classes $F$ and $Z$. It suffices to give these lifts at the fixed points $X_i$. For a single point, we have $H^\bullet_T(\{*\}, \mathbb{Q}) \cong K^0_T(\{*\}) \otimes \mathbb{Q}= \mathbb{Q}[s, t]$. It turns out that there is an equivariant class lifting $F$ which restricts to $s^{-1}$ at $X_1$ and $X_2$ and zero otherwise, and there is one lifting $Z$ which restricts to $t^{-1}$ at $X_1$ and $X_4$ and zero otherwise. Then the point class $\mathbf{p}$ is simply the product of these classes. 

  Now we can compute all classes $\ch_i(\gamma)$, but to verify the conjecture, one also needs the operator $\mathcal{L}_k$. For this, we also need a Künneth decomposition of the diagonal $X \to X \times X$. This can be computed to be $\mathbf{p} \otimes 1 + F \otimes Z + Z \otimes F + 1 \otimes \mathbf{p}$, for example by the method of \cite[Sec. 2.1.6]{eisenbud16}.

  For example, to verify that the conjecture holds for $\mathcal{L}_2\ch_2(Z)$, one needs to integrate
  \[R_2\ch_2(Z) = 6\ch_4(Z), \qquad T_2\ch_2(Z)\quad \text{and} \quad S_2\ch_2(Z) = 3(\ch_2(\mathbf{p})\ch_2(Z) + \ch_3(\mathbf{p})\ch_1(Z)) \]
  We refrain from writing down all intermediate steps, and simply give the result: $-\frac18$, $-\frac{1}{4}$ and $\frac38$. These sum up to zero, as required.
\end{example}

\subsection{Finding all equivariant sheaves}

Above we assumed we already had access to all equivariant stable sheaves with certain fixed invariants. Now we explain how to find these. We assume that we have a fixed polarisation $H$. Later we will explain how to find all possible $H$, so that one can apply this procedure to each one. The procedure can be divided into two steps: first find all equivariant vector bundles, which we then use to find all equivariant torsion-free sheaves. We first explain how to reduce to vector bundles.

\subsubsection{Reducing to vector bundles}

Note that for any $\mu$-stable sheaf $E$ with fixed invariants $r$, $c_1$ and $c_2$, the canonical map $E \to E^{\vee\vee}$ is an embedding of $E$ into a $\mu$-stable vector bundle of rank $r$, $c_1(E^{\vee\vee}) = c_1$ and $c_2(E^{\vee\vee}) \leq c_2$. But we also have the lower bound $c_2(E^{\vee\vee}) \geq\frac{(r - 1)c_1^2}{2r}$ given by the Bogomolov inequality. Thus, if we are able to find a list of all stable equivariant vector bundles $F$ satisfying $\rk F = r$, $c_1(F) = c_1$ and $\frac{(r - 1)c_1^2}{2r} \leq c_2(F) \leq c_2$, then $E^{\vee\vee}$ must be in this list. Furthermore, if we have such an equivariant vector bundle $F$ and any equivariant torsion-free sheaf $E$ with $E^{\vee\vee} = F$, then $E$ is $\mu$-stable. Thus it suffices to, given $F$, find all equivariant torsion-free $E$ with the right invariants such that $E^{\vee\vee} = F$.

Consider the following operation on $F$. We consider a maximal cone $\sigma$ and the $\sigma$-family $\hat{F}$ of $F$ on $U_{\sigma}$. We consider a point $(m, n)$ such that $\hat{F}(m-1, n) + \hat{F}(m, n-1)$ is a proper subspace of $\hat{F}(m, n)$. We then replace $\hat{F}(m, n)$ by a subspace of codimension one, which contains $\hat{F}(m-1, n)+\hat{F}(m, n-1)$. This gives a sub-equivariant sheaf $F'$ of $F$ which has its $c_2$ increased by one. Since $(F')^{\vee\vee} = F^{\vee\vee}$, $F'$ is still $\mu$-stable and hence Gieseker stable. We have constructed $F'$ in such a way that there is a short exact sequence
\begin{equation}
  \label{eq:torsion-free-exact-sequence}
  0 \to F' \to F \to \chi \cdot \mathcal{O}_{p} \to 0.
\end{equation}
Here $p$ is the fixed points of $X$ corresponding to the maximal cone $\sigma$ and $\chi$ is the character of $F(m, n)$.

All equivariant sheaves $E$ with $E^{\vee\vee} = F$ are obtained by applying this operation sufficiently many times so that the resulting sheaf has the desired $c_2$. If the difference in $c_2$'s is only one, this can easily be seen by considering the quotient. Otherwise one can use induction.

\begin{example}
  On $\mathbb{P}^2$ there are 1, 3 and 3 equivariant vector bundles of rank 2 with $c_1 = 1$ and $c_2$ equal to respectively 1, 2 and 3. There are six ways to degenerate a stable equivariant bundle with $c_2 = 2$ to a stable equivariant sheaf with $c_2 = 3$. The single equivariant bundle with $c_1 = 1$ degenerates to a total of 27 stable equivariant sheaves with $c_2 = 3$. Hence the moduli space has 48 isolated fixed points, the highest of the examples in this paper. See Appendix \ref{sec:data-obtained-from} for a complete list.
  
  We make this more concrete in a specific example. There is a stable equivariant vector bundle $E$ with $c_2 = 2$ whose K-theory classes at the fixed points of $\mathbb{P}^2$ are $(st^{-1} + ts^{-1}, st^{-1} + t,  ts^{-1} + s)$. To perform the operation described above, we can consider the first chart and the weight space of $st^{-1}$. This space is one-dimensional, so there is no choice in picking a codimension-one subspace. We then obtain an equivariant stable torsion-free sheaf $E'$ which is not locally free.

  In this paper we are mainly interested in the K-theory class of $E'$. For this, we can use the exact sequence \eqref{eq:torsion-free-exact-sequence}. It is not difficult to see that the K-theory class of $\chi \cdot \mathcal{O}_p$ is $\chi \cdot (1 - s - t + st)$. So we can obtain the K-theory class of $E'$ by replacing $st^{-1}$ by $st^{-1}(s + t - st)$ in the K-theory class of $E$. This computation also works if we had chosen the weight space $ts^{-1}$. It also works if we had chosen one of the other two charts, except that the K-theory class of $\chi \cdot \mathcal{O}_p$ is then $\chi \cdot(1 - s^{-1} - ts^{-1} + ts^{-2})$ or $\chi \cdot (1 - t^{-1} - st^{-1} + st^{-2})$ respectively.
\end{example}

\begin{example}
  \begin{figure}
  \centering
  \begin{tikzpicture}[scale = .5]
        \draw(0,0) -- (2,0);
    \draw(0,-1) -- (2,-1);
    \draw [dashed] (2,0) -- (4,0);
    \draw [dashed] (2,-1) -- (4,-1);
    \draw (4,0) -- (6,0);
    \draw (4,-1) -- (6,-1);
    \draw(0,0) -- (0,-1);
    \draw(6,0) -- (6,-1);

    \draw(0,0) -- (0,2);
    \draw[dashed] (0,2) -- (0,4);
    \draw(0,4) -- (0,6);
    \draw(-1,0) -- (-1,2);
    \draw[dashed] (-1,2) -- (-1,4);
    \draw(-1,4) -- (-1,6);
    \draw(-1,0) -- (0,0);
    \draw(-1,6) -- (0,6);

    \draw(0,7) -- (2,7);
    \draw(0,6) -- (2,6);
    \draw [dashed] (2,7) -- (4,7);
    \draw [dashed] (2,6) -- (4,6);
    \draw (4,7) -- (6,7);
    \draw (4,6) -- (6,6);
    \draw(0,7) -- (0,6);
    \draw(6,7) -- (6,6);

    \draw(7,0) -- (7,2);
    \draw[dashed] (7,2) -- (7,4);
    \draw(7,4) -- (7,6);
    \draw(6,0) -- (6,2);
    \draw[dashed] (6,2) -- (6,4);
    \draw(6,4) -- (6,6);
    \draw(6,0) -- (7,0);
    \draw(6,6) -- (7,6);

  \end{tikzpicture}
  \qquad \quad
\begin{tikzpicture}[scale = .5]
        \draw(0,0) -- (2,0);
    \draw(0,-1) -- (2,-1);
    \draw [dashed] (2,0) -- (4,0);
    \draw [dashed] (2,-1) -- (4,-1);
    \draw (4,0) -- (6,0);
    \draw (4,-1) -- (6,-1);
    \draw(0,0) -- (0,-1);
    \draw(6,0) -- (6,-1);

    \draw(0,0) -- (0,2);
    \draw[dashed] (0,2) -- (0,4);
    \draw(0,4) -- (0,6);
    \draw(-1,0) -- (-1,2);
    \draw[dashed] (-1,2) -- (-1,4);
    \draw(-1,4) -- (-1,6);
    \draw(-1,0) -- (0,0);
    \draw(-1,6) -- (0,6);

    \draw(0,7) -- (2,7);
    \draw(0,6) -- (2,6);
    \draw [dashed] (2,7) -- (4,7);
    \draw [dashed] (2,6) -- (4,6);
    \draw (4,7) -- (6,7);
    \draw (4,6) -- (6,6);
    \draw(0,7) -- (0,6);
    \draw(6,7) -- (6,6);

    \draw(7,0) -- (7,2);
    \draw[dashed] (7,2) -- (7,4);
    \draw(7,4) -- (7,5.5);
    \draw(6,0) -- (6,2);
    \draw[dashed] (6,2) -- (6,4);
    \draw(6,4) -- (6,6);
    \draw(6,0) -- (7,0);
    \draw(6,6) -- (6.5,6);

    \draw(6.5,6) -- (6.5,5.5);
    \draw(6.5,5.5) -- (7,5.5);

  \end{tikzpicture}
\qquad\quad
  \begin{tikzpicture}[scale = .5]
        \draw(0,0) -- (2,0);
    \draw(0,-1) -- (2,-1);
    \draw [dashed] (2,0) -- (4,0);
    \draw [dashed] (2,-1) -- (4,-1);
    \draw (4,0) -- (6,0);
    \draw (4,-1) -- (6,-1);
    \draw(0,0) -- (0,-1);
    \draw(6,0) -- (6,-1);

    \draw(0,0) -- (0,2);
    \draw[dashed] (0,2) -- (0,4);
    \draw(0,4) -- (0,6);
    \draw(-1,0) -- (-1,2);
    \draw[dashed] (-1,2) -- (-1,4);
    \draw(-1,4) -- (-1,6);
    \draw(-1,0) -- (0,0);
    \draw(-1,6) -- (0,6);

    \draw(0,7) -- (2,7);
    \draw(0,6) -- (2,6);
    \draw [dashed] (2,7) -- (4,7);
    \draw [dashed] (2,6) -- (4,6);
    \draw (4,7) -- (7,7);
    \draw (4,6) -- (6,6);
    \draw(0,7) -- (0,6);

    \draw(7,0) -- (7,2);
    \draw[dashed] (7,2) -- (7,4);
    \draw(7,4) -- (7,7);
    \draw(6,0) -- (6,2);
    \draw[dashed] (6,2) -- (6,4);
    \draw(6,4) -- (6,6);
    \draw(6,0) -- (7,0);

  \end{tikzpicture}
  \caption{Three equivariant sheaves on a Hirzebruch surface.} \label{fig:hirz-three-bundles}
\end{figure}
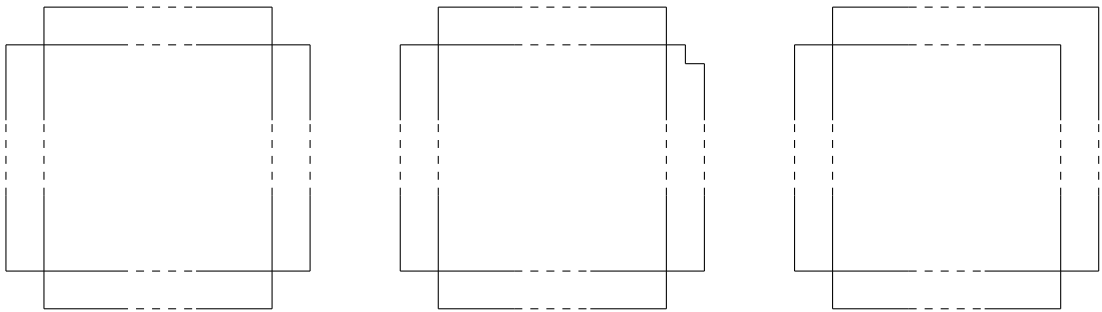


  We show how to represent the equivariant torsion-free sheaves pictorially on a Hirzebruch surface. Recall from Example \ref{ex:p2-tan-bundle} that one can represent a rank 2 vector bundle on $\mathbb{P}^2$ by a triangle with strips attached. For a Hirzebruch surface the picture is the same, except that we now have a square instead of a triangle (essentially because on a Hirzebruch surface we have four toric charts). Each of the corners represents one of the charts and each side represents a ray of the cone (i.e., the gluing condition $(\star)$). The first picture in Figure \ref{fig:hirz-three-bundles} represents an equivariant vector bundle on $\mathbb{F}_a$. Consider the chart in the upper-right corner. In this chart, there are two places where we can apply the operation described above, namely in the upper-right corners of the two strips. Suppose we choose the right strip. Then we lower the dimension of the weight space in the upper-right corner from one to zero. Hence we obtain the second picture in Figure \ref{fig:hirz-three-bundles}. The small square we removed represents the single weight space we removed.
\end{example}

\subsection{Finding all equivariant vector bundles}

We describe how to find all equivariant stable vector bundles. We focus on the rank 2 case on a Hirzebruch surface $\mathbb{F}_a$, but the other cases are similar. The fan of a Hirzebruch surface has four two-dimensional cones and four rays and can be found in \cite[Sec. 1.1]{fulton93_toric}. We number the rays in clockwise order, starting with the rightmost ray. In this case an equivariant vector bundle $E$ is described by four finite full flags of $\mathbb{C}^2$, one for each ray. Thus, for each ray $\rho_i$ we have a sequence which looks like the following:
\[
  \ldots \subseteq 0  \subseteq \ldots \subseteq V^i \subseteq \ldots \subseteq V^i \subseteq \ldots \subseteq \mathbb{C}^2 \subseteq \ldots
\]
Such a flag depends on three choices: the one-dimensional subspace $V^i$, a number $\delta_i$ which is the number of occurences of $V^i$ and a number $A^i$ which is the weight where $\mathbb{C}^2$ first occurs. However, different choices may lead to the same vector bundle, since we have the freedom of tensoring with a character from $T = \mathbb{G}^2_m$ and we can apply a linear automorphism of $\mathbb{C}^2$ to obtain different $V^i$. The action of a character of $T$ only changes the $A^i$. To be precise, the action of $s^nt^m$ changes $(A^1, A^2, A^3, A^4)$ to $(A^1+n, A^2 - m, A^3 - n, A^4 + m)$. So we can eliminate the ambiguity of the choice of character by fixing (for example) $A^2 = A^3 = 0$. The ambiguity in the choice of $V^i$ is more subtle. For the moment, let us assume that all the $V^i$ are distinct. The computation does not depend on the specific choice of $V^i$. Later we will deal with the other cases. Pictorially, this means we are in the situation corresponding to the first picture in Figure \ref{fig:hirz-three-bundles}.

We will determine all possible numerical invariants. These are the solutions to a certain set of equalities and inequalities, determined by fixing $c_1$ and $c_2$, and the inequalities coming from stability (Prop~\ref{prp:stability-ineq}). By considering a resolution, we can find explicit formulas for the first and second Chern character of $E$. The formulas also depends on $a$ (recall we work on $\mathbb{F}_a$).
\begin{equation*}
  \ch_1(E)  = (\delta_1 + \delta_3 + a\delta_2 - 2A^1)F + (\delta_2 + \delta_4 - 2A^4)Z
\end{equation*}

Hence if we have fixed $\Delta = fF + zZ$, then we get the equations $f = \delta_1 + \delta_3 + a\delta_2 - 2A^1$ and $z = \delta_2 + \delta_4 - 2A^4$. The formula for the second Chern character is

\begin{equation}
  \label{eq:chern-2-hirzebruch}
  \ch_2(E) = \frac14 \left(
  a(\delta_2^2 - \delta_4^2 - z^2) - 2(\delta_1\delta_2 + \delta_2 + \delta_3 + \delta_3 \delta_4 + \delta_4 \delta_1) + 2fz \right)
\end{equation}

This formula requires the $V_i$ to be distinct.

Next we need to consider the stability inequalities from Prop. \ref{prp:stability-ineq}. To compute the degrees, we need the intersection numbers $\deg_1 = \deg_3 = H.F$, $\deg_2 = H.(aF + Z)$ and $\deg_4 = H.Z$. The proposition tells us that we need to consider all possible one-dimensional subspaces $W$, but in fact it suffices to consider $W = V^i$ for some $V^i$. Indeed, otherwise the inequality is trivial as $w^i_j$ is always zero in that case. Then the stability inequalities become
\begin{equation}
  \label{eq:stability-hirz}
  2\delta_i \deg_i < \sum_{j = 1}^4 \delta_j \deg_j
\end{equation}
for each $i$. Here we also use that the $V^i$ are distinct. Finally, we need the trivial inequalities $\delta_i \geq 0$. We then find all solutions for this system. This is elementary, but still difficult. We made use of the automated theorem prover Z3 \cite{de_Moura_2008} to find all solutions.

Now we deal with the problem of choice of the vector spaces $V^i$. It turns out that in the numbers we obtained above, at least one $\delta_i$ was always zero. Thus we only had to choose three distinct vector spaces. But then there is no choice at all: if we choose different subspaces $W^i$ then there is a linear automorphism of $\mathbb{C}^2$ mapping $V^i$ to $W^i$. So the resulting equivariant stable vector bundle does not depend on the choice of $V^i$. This also implies that the fixed point locus $M^T$ is isolated, since it only depends on discrete numerical data. If all $\delta_i$ are positive, we would have had a higher-dimensional component of $M^T$. However, in this paper, it turns out that this behaviour does not occur and we only have to deal with cases where $\dim M^T = 0$.

\subsection{Degenerate cases}

In the previous computation we assumed that the $V^i$ were all distinct. Here we explain how to find the vector bundles if that were not true. The stability inequality Prop. \ref{prp:stability-ineq} imply the following: there must always exist at least three $i$ for which $\delta_i > 0$ and for which the $V^i$ are distinct. Then there are the following cases to consider: two adjacent $V^i$ are equal (e.g. $V^1 = V^4$) or two opposing $V^i$ are equal (e.g. $V^1 = V^3$). We explain the changes.

The formula~(\ref{eq:chern-2-hirzebruch}) for the second Chern character changes if two adjacent $V^i$ are equal. Explicitely, if $V^1 = V^4$, one should add a term $\delta_1\delta_4$ to the formula, and similar if other adjacent $V_i$ are equal. This is similar for the other cases. The other thing that changes is the stability conditions, both in the case of adjacent and opposing coincidence of the $V^i$. Indeed, suppose $V^1 = V^3$, then taking $W = V^1 = V^3$ in Prop.~\ref{prp:stability-ineq} gives the more restrictive inequality
\begin{equation}
  \label{eq:stability-hirz-degen}
  2(\deg_1\delta_1 + \deg_3 \delta_3) = \sum_{j = 1}^4 \deg_j\delta_j.
\end{equation}
One needs to add this inequality to the system we already had in \eqref{eq:stability-hirz}. We can solve this system in a similar way to the other case. Because there are only three different vector spaces $V^i$ from the start, we can argue in the same way as before to ensure that $M^T$ is isolated.

Pictorially, the condition that $V^1 = V^4$ corresponds to the third picture in Figure \ref{fig:hirz-three-bundles} where we have filled in the corner. Indeed, if we consider this chart, then recall that for a vector bundle $E$ we should have $\hat{E}(m, n) = \hat{E}(m, \infty) \cap \hat{E}(\infty, n)$. If $(m, n)$ lies in this corner, this intersection is $V^1 \cap V^4$. In a generic situation, $V^1 \neq V^4$, so that the intersection is zero. But in this special case, the intersection equals $V^1 = V^4$, so it is one-dimensional.

The picture also explains why the second Chern character changes. The little square we added by setting $V^1 = V^4$ has size $\delta_1 \times \delta_4$. It turns out that the Chern character only depends on the dimension of the weight spaces, see \cite[Prop. 3.6]{kool09_euler_charac_modul_spaces_torsion}. Changing a single weight space by one dimension changes the second Chern character by one and the other Chern characters stay the same. Therefore, a square of dimension $\delta_1 \times \delta_4$ represents a change of $\delta_1\delta_4$ in the second Chern character. Note that in the situation $V^1 = V^3$, we do not add such a small square and consequently the second Chern character does not change.

\subsection{Finding bundles on \texorpdfstring{$\mathbb{P}^2$}{P\^{}2}}

Finding stable equivariant vector bundles for $\mathbb{P}^2$ is similar to what we have just described. Recall that the fan for $\mathbb{P}^2$ has three rays. Thus the rank 2 case becomes even easier, since there are only three vector spaces $V^i$ from the start, which makes it easier to argue that $M^T$ is isolated. However, the rank 3 and 4 cases are much more involved. We briefly discuss the rank 3 case. Now, our flags look like
\[
  \ldots \subseteq 0 \subseteq V^i \subseteq \ldots \subseteq V^i \subseteq W^i \subseteq \ldots \subseteq W^i \subseteq \mathbb{C}^3 \subseteq \ldots
\]
where $V^i$ is one-dimensional and $W^i$ is two-dimensional. Note that we require three numbers to describe such a flag, with $A^i$ as before and $\delta_i$ and $\epsilon_i$ for the number of occurences of $V^i$ and $W^i$, respectively. Having two spaces in our flag gives a great additional complexity in the ambiguity of choices of these subspaces. Also, classifying the possible coincidences is more complicated. This was done by Klyachko in a preprint \cite{Klyachko_preprint}, but see also \cite[Sec. 4.2]{kool09_euler_charac_modul_spaces_torsion}. Recall that one- and two-dimensional subspaces of $\mathbb{C}^3$ are the points and lines in $\mathbb{P}^ 2$. Thus we may picture the $V^i$ and $W^i$ as three lines with three points on them (indicating that the $V^i$ are subspaces of the $W^i$). Then stability ensures that the possible configurations are the ones pictured in Figure~\ref{fig:coincidences}.

The resulting inequalities are also much harder than in the Hirzebruch case. We applied a mix of calculations by hand and the solver Z3 to find all solutions. In the end, we can argue as before that since a sufficient number of $\delta_i$ and $\epsilon_i$ vanish, the isomorphism class of the obtained vector bundle $E^i$ does not depend on the choice of subspaces $V^i$ and $W^i$, as any two choices can be related by a suitable automorphism of $\mathbb{C}^3$. This implies that $M^T$ is isolated, as before.

\begin{figure}
  \centering
  \begin{tikzpicture}[scale=0.12]
    \draw (0,0) -- (25,25);
    \draw (0,5) -- (40,5);
    \draw(15,25)-- (40,0);

    \node at (12.5,12.5) [circle, fill, inner sep=1.5pt] {};
    \node at (27.5,12.5) [circle, fill, inner sep=1.5pt] {};
    \node at (20,5) [circle, fill, inner sep=1.5pt] {};
  \end{tikzpicture}\hspace{30pt}
  \begin{tikzpicture}[scale=0.15]
    \draw (-15,0) -- (15,0);
    \draw (10,10) -- (-10,-10);
    \draw(10,-10)-- (-10,10);

    \node at (7.5,7.5) [circle, fill, inner sep=1.5pt] {};
    \node at (-7.5,7.5) [circle, fill, inner sep=1.5pt] {};
    \node at (9.5,0) [circle, fill, inner sep=1.5pt] {};
  \end{tikzpicture}\\[10pt]
  \begin{tikzpicture}[scale=0.12]
    \draw (0,0) -- (25,25);
    \draw (0,5) -- (40,5);
    \draw(15,25)-- (40,0);

    \node at (5,5) [circle, fill, inner sep=1.5pt] {};
    \node at (27.5,12.5) [circle, fill, inner sep=1.5pt] {};
    \node at (20,5) [circle, fill, inner sep=1.5pt] {};
  \end{tikzpicture}\hspace{30pt}
  \begin{tikzpicture}[scale=0.10]
    \draw (-13.5,-13.5) -- (25,25);
    \draw (0,5) -- (40,5);
    \draw(15,25)-- (40,0);

    \draw[dashed] (-20,-15)--(40,15);
    
    \node at (-10,-10) [circle, fill, inner sep=1.5pt] {};
    \node at (30,10) [circle, fill, inner sep=1.5pt] {};
    \node at (20,5) [circle, fill, inner sep=1.5pt] {};
  \end{tikzpicture}

  \caption{Possible coincidences}
  \label{fig:coincidences}
\end{figure}
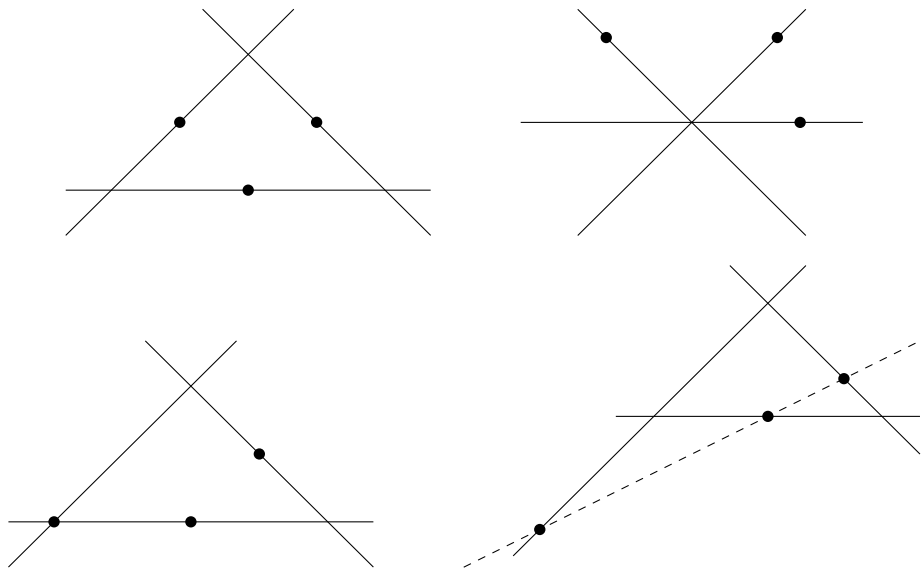


The classification of bundles of rank 4 on $\mathbb{P}^2$ is incomplete, as we found it already too difficult to classify the possible coincidences of the subspaces in the flags in this case. We found 13 bundles by looping over the $\delta_i$ in the non-degenrate case where there are no special coincidences. We are nevertheless confident that there are no more bundles because the Atiyah-Bott calculation gave numbers as a result. Usually when one forgets a bundle, the answer is a not a number, but a more complicated expression in the equivariant parameters.

\subsection{Dependence on the polarisation}

Both $\mu$-stability and Gieseker stability depend on the choice of a polarisation. We are interested in verifying the conjecture for all possible choices of polarisation. For $\mathbb{P}^2$ this is not an issue, since the only polarisations are $\mathcal{O}(n)$ for $n > 0$ and the stability condition does not change if you scale a polarisation by a positive number.

But for Hirzebruch surfaces this is an issue. Note that $\Pic(\mathbb{F}_a) = \Num(\mathbb{F}_a) = \mathbb{Z}F \oplus \mathbb{Z}Z$ with intersections $F.F = 0$, $F.Z = 1$ and $Z.Z = -a$. The ample cone is the set of line bundles $L$ such that $L.F > 0$ and $L.Z > 0$. $\Num(\mathbb{F}_a)$ contains a finite set of hyperplanes, called \emph{walls}, which divide the ample cone into connected components, called \emph{chambers}. This material can be found in \cite[Sec. 4.C]{huybrechts10} and was developed by Z. Qin \cite{Qin1993EquivalenceCO}. If $F$ is a rank 2 sheaf $F$ with $c_1(F)$ not divisible by 2, then $F$ is $\mu$-stable with respect to a polarisation $H$ iff it is $\mu$-stable with respect to any polarisation in the chamber of $H$. Hence the moduli space $M = M^H_X(r, \Delta, c_2)$ does not change if we pick another polarisation in the same chamber. We do not need to consider polarisations $H$ that lie on walls: if the moduli space $M$ changes if we pick another polarisation $H'$ that lies close to $H$ and is in a chamber, then this change happens because of the existence of proper $\mu$-semi-stable sheaves with respect to $H$. But we assumed no such sheaves existed.

Now we explain how to find all walls. For a general surface $X$ and fixed $r$, $\Delta$ and $c_2$, consider all the $\xi \in \Num(X)$ such that
\[ \frac{((r^2 - 1)\Delta^2 - 2rc_2)r^2}{4} \leq \xi^2 < 0.\]
Then the walls are given by the hyperplanes $\{\xi\}^\perp$. Furthermore, if $X$ is a Hirzebruch surface, there are only a finite number of such $\xi$. To verify Conjecture~\ref{sec:strong-conjecture} it suffices to choose a polarisation from each of the chambers and apply the algorithm described in the beginning of this section.

\begin{example}
  In our examples, there are two cases where the choice of polarisation is relevant, both for $\mathbb{F}_0$ and $r = 2$. In the first case, we took $\Delta = F+Z$ and $c_2 = 2$ and in the second case we had $\Delta = F$ had and $c_2 =2$. In both cases there are several chambers, but only two possible variants of $M^T$. The second example is the most interesting. There are two equivariant bundles which are stable independent of the polarisation. But for one choice, one obtains for additional equivariant vector bundles, while for the other, one obtains 20 equivariant torsion-free sheaves, only four of which are vector bundles! Thus even the topological Euler characteristic depends on the choice of polarisation (as is well known). All the sheaves are listed in Appendix \ref{sec:data-obtained-from}.
\end{example}

\subsection{A deformation argument}

Using the algorithm above, we have verified Conjecture \ref{sec:strong-conjecture} explicitely for the Hirzebruch surfaces $\mathbb{F}_a$ with $a = 0, 1, 2$, $\Delta = F$, $Z$ and $H+Z$ and $c_2$ the minimal choice such that the Bogomolov inequality is satisfied. We will now use a deformation argument to deduce Conjecture \ref{sec:strong-conjecture} for any $a$ with $\Delta$ and $H$ such that $\Delta.H$ is odd and $c_2$ again the minimal choice. In fact the case $a = 2$ is redundant, as it also follows from the deformation argument.

\begin{proposition}
  \label{prp:family-f-a}
  Let $n$ be a natural number. There is a smooth family of surfaces $\mathcal{F}$ over $\mathbb{A}^1$ with fibre $\mathbb{F}_{2n}$ above zero and fibre $\mathbb{F}_0$ over all other closed points.

  Similarly, there is a family $\mathbb{F}$ with fibre $\mathcal{F}_{2n+1}$ above zero and fibre $\mathbb{F}_1$ over all other points.
\end{proposition}
\begin{proof}
  We work on $\mathbb{P}^1$. Consider the subsheaf $L_1$ of $\mathcal{O}$ of functions that have a zero at 0 of order at least $n$. Similarly, we consider the subsheaf $L_2$ of functions that have a zero at $\infty$ of order at least $n$. The sum of the inclusions $L_1 \oplus L_2 \to \mathcal{O}$ is a surjection whose kernel consists of the functions vanishing at both zero and $\infty$ of order at least $n$. Of course, $L_1 \cong L_2 \cong \mathcal{O}(-n)$ and the kernel is $\mathcal{O}(-2n)$. Hence this gives a short exact sequence
  \[
    0 \to \mathcal{O}(-2n) \to \mathcal{O}(-n) \oplus \mathcal{O}(-n) \to \mathcal{O} \to 0.
  \]
  Let $\mathbb{A}^1$ be the line in $\Ext^1(\mathcal{O}, \mathcal{O}(-2n))$ through zero and the above extension. Then there is a $\mathbb{A}^1$-flat family $\mathcal{E}$ on $\mathbb{P}^1 \times \mathbb{A}^1$, which is $\mathcal{O}(-n) \oplus \mathcal{O}(-n)$ over every nonzero fibre and $\mathcal{O}(-2n) \oplus \mathcal{O}$ over zero \cite[Rem. 3.5]{LANGE1983101}. Let $\mathcal{F}$ be the projectivisation of $\mathcal{E}$. Then $\mathcal{F}$ satisfies the conditions of the lemma.

  For the odd case a similar argument works.
\end{proof}

\begin{lemma}
  \label{lmm:extend-line-bundles}
  In the previous proposition, for every line bundle $L$ on $\mathbb{F}_{2n}$ (resp. $\mathbb{F}_{2n+1}$), there is an $\mathbb{A}^1$-flat family of line bundles $\mathcal{L}$ on $\mathcal{F}$ such that $\mathcal{L}_0 = L$. 
\end{lemma}
\begin{proof}
  It suffices to show this for the generators $F$ and $Z$ of $\Pic(\mathbb{F}_{2n})$. These are defined (as divisors) as a fibre of $\pi : \mathbb{F}_{2n} \to \mathbb{P}^1$ and a section of $\pi$ respectively. Thus $F$ clearly extends to the family because we can simply take a family of fibres of $\mathcal{F} \to \mathbb{A}^1 \times \mathbb{P}^1$. For the section, note that the family $\mathcal{E}$ on $\mathbb{P}^1 \times \mathbb{A}^1$ comes with a surjection $\mathcal{E} \to \mathcal{O}$ because $\mathcal{E}$ is an extension. This defines a section \cite[Prop. 7.12]{hartshorne77_algeb_geomet} and we simply take the image of this section as our extension of $Z$. Again the odd case is similar.
\end{proof}

\begin{proposition}
  Let $r = 2$. Choose a line bundle $\Delta$ on $\mathbb{F}_a$. Let $c_2$ be the smallest number such that the Bogomolov inequality is satisfied with $r = 2$. Then Conjecture \ref{sec:strong-conjecture} holds for any polarisation $H$ on $\mathbb{F}_a$ such that $H.\Delta$ is odd.
\end{proposition}

\begin{proof}
  The idea is to use Prop. \ref{prp:deformation-invariance}. We have constructed a smooth family $\mathcal{F} \to \mathbb{A}^1$ in Lemma \ref{prp:family-f-a}. By Lemma \ref{lmm:extend-line-bundles}, we can extend the polarisation $H$ and the line bundle $\Delta$
  to families $\mathcal{H}$ and $\mathcal{L}$ on $\mathcal{F}$. We can also extend $c_2$ by taking $(c_2 H) . Z$ for example. The intersection numbers $\mathcal{H}_t. \mathcal{L}_t$ are the same for every fibre. This is also true for other intersection numbers. Hence we deduce:
  \begin{enumerate}
  \item $\mathcal{H}_t.\mathcal{L}_t$ is odd for each closed point $t \in \mathbb{A}^1$.
  \item $\mathcal{H}_t$ is ample for each closed point $t \in \mathbb{A}^1$, hence is ample relative to $\pi$ \cite[Sec. 1.7]{lazarsfeld04_posit}
  \item $c_2$ is the minimal $c_2$ such that the Bogomolov inequality holds with $r = 2$ and $c_1 = \mathcal{L}_t$.
  \end{enumerate}

  We wish to verify the Virasoro constraints for the central fibre, so by Prop. \ref{prp:deformation-invariance} we can also do so for another fibre, which is always $\mathbb{F}_0$ or $\mathbb{F}_1$. We treat the first case, as the second is similar. For concreteness, we pick the fibre over $1 \in \mathbb{A}^1$ and verify Conjecture \ref{sec:strong-conjecture} there. Since $\mathcal{H}_1.\mathcal{L}_1$ is odd, $\mathcal{L}_1$ is not a multiple of 2. Hence there is a line bundle $L'$ on $\mathbb{F}_0$ such that $\mathcal{L}_1 + 2 L'$ is either $F$, $Z$ or $F+Z$. Now we use that when $\mu$-stability and Giesker stability coincide, that there is an isomorphism
  \[M^{\mathcal{H}_1}_{\mathbb{F}_0}(2, \mathcal{L}_1, c_2) \to M^{\mathcal{H}_1}_{\mathbb{F}_0}(2, \mathcal{L}_1 + 2L', c_2 + L'.(\mathcal{L}_1 + L')),\] which is given by $[E] \mapsto [E \otimes L']$ on closed points. On this second space, $c_2$ is still the minimal number such that the Bogomolov inequality holds. Hence we already know the Virasoro constraints for this case, by our explicit computation. 
\end{proof}

\begin{remark}
  One might wonder if the moduli space for $\mathbb{F}_a$ is actually nonempty in the above proposition. This is true. The relative moduli space $\mathcal{M} \to \mathbb{A}^1$ is projective over $\mathbb{A}^1$, so its image is closed. For $t \neq 0$, we know that the moduli space is not empty, because we have found explicit sheaves in the fixed point locus. But then the image of $\mathcal{M}$ must be all of $\mathbb{A}^1$, so the fibre above $t = 0$ cannot be empty.
\end{remark}

\appendix

\section{Data obtained from computations}
\label{sec:data-obtained-from}

In this section we will give the explicit results from the computations. We fix some notation. We work over a toric surface $X$. We denote the moduli space of stable sheaves on $X$ with rank $r$, determinant $\Delta$ and second Chern class $c_2$ by $M$. We denote the representation ring of $T = \mathbb{G}_m^2$ by $\mathbb{Z}[s, t]$. For projective space, we let $X_1 = (1:0:0)$, $X_2 = (0:1:0)$ and $X_3 = (0:0:1)$. We let $H$ be the hyperplane class. For the Hirzebruch surface $\mathbb{F}_a$, we let $X_1$, $X_2$, $X_3$ and $X_4$ be the fixed points where the tangent space is given by $s^{-1} + t^{-1}$, $s^{-1} + t$, $s + ts^a$ and $s + t^{-1}s^{-a}$ respectively (in standard coordinates). Furthermore, recall that $\Pic \mathbb{F}_a = \mathbb{Z}F + \mathbb{Z}Z$ where $F$ is a fibre of the projection $\mathbb{F}_a \to \mathbb{P}^1$ and $Z$ is a section of this projection. We verified the conjecture for all polarisations $H$ not on a wall (or equivalently, we picked some $H$ from each chamber). When the choice of $H$ is relevant, we write which choice we made, but otherwise we will not mention $H$.

For each of the cases we have considered, we write down the explicit K-theoretic representations of all equivariant sheaves at the fixed points. We also write down $\int_M R_kD$, $\int_M T_kD$ and $\int_MS_kD$ for $k = \dim M$, $\dim M - 1$ and $\dim M - 2$ and all $D$ of the correct degree. Even though we have checked it for all $k$ and $D$ but we do not give all the data for reason of space: we did 1677 independent checks in total.

\setlength{\extrarowheight}{2pt}

\vspace{1em}
\begin{centering}
  \begin{tabularx}{\textwidth}{
      |>{\centering\arraybackslash}X
      |>{\centering\arraybackslash}X
      |>{\centering\arraybackslash}X
      |>{\centering\arraybackslash}X
      |>{\centering\arraybackslash}X |}
    \hline
    $X = \mathbb{P}^2$ & $r = 2$ & $\Delta = \mathcal{O}(1)$ & $c_2 = 1$ & $\dim M = 0$
    \\\hline
\end{tabularx}
\end{centering}

The moduli space is zero-dimensional and there is only one equivariant sheaf. Therefore the Virasoro constraints are automatic. We give the (K-theory class of) this sheaf below.

\begin{center}
\begin{tabularx}{.8\textwidth}{
    |>{\centering\arraybackslash}X
    |>{\centering\arraybackslash}X
    |>{\centering\arraybackslash}X |}
  \hline $F|_{X_1}$ & $F|_{X_2}$ & $F|_{X_3}$ \\ \noalign{\hrule height 1.5pt}
  $t^{-1} + s^{-1}$ & $1 + t^{-1}$ & $1 + s^{-1}$ \\ \hline
\end{tabularx}
\end{center}


\vspace{1em}
\begin{centering}
  \begin{tabularx}{\textwidth}{
      |>{\centering\arraybackslash}X
      |>{\centering\arraybackslash}X
      |>{\centering\arraybackslash}X
      |>{\centering\arraybackslash}X
      |>{\centering\arraybackslash}X |}
    \hline
    $X = \mathbb{P}^2$ & $r = 3$ & $\Delta = \mathcal{O}(1)$ & $c_2 = 2$ & $\dim M = 2$
    \\\hline
\end{tabularx}
\end{centering}

\begin{center}
\begin{tabularx}{.8\textwidth}{
    |>{\centering\arraybackslash}X
    |>{\centering\arraybackslash}X
    |>{\centering\arraybackslash}X |}
  \hline $F|_{X_1}$ & $F|_{X_2}$ & $F|_{X_3}$ \\ \noalign{\hrule height 1.5pt}
  $st^{-1} + 1 + ts^{-1}$ & $1 + st^{-1} + t$ & $1 + ts^{-1} + s$ \\ \hline
  $1 + t^{-1} + ts^{-1}$ & $1 + t^{-1} + t$ & $1 + t + s^{-1}$ \\ \hline
  $1 + s^{-1} + st^{-1}$ & $1 + s + t^{-1}$ & $1 + s^{-1} + s$ \\ \hline
\end{tabularx}
\end{center}

\begin{center}
  \begin{tabularx}{\textwidth}{
      |>{\centering\arraybackslash}X
      |>{\centering\arraybackslash}X
      |>{\centering\arraybackslash}X
      |>{\centering\arraybackslash}X|}
    \hline
    $D$ & $\int_M R_kD$ & $\int_M T_kD$ & $\int_M S_kD$ \\ \hline \hline
    1 &  $0$ & $-10/3$ & $10/3$ \\ \hline
$\ch_2(H)$ &  $7/9$ & $-7/9$ & $0$ \\ \hline
$\ch_3(1)$ &  $-17/18$ & $-49/54$ & $50/27$ \\ \hline
$\ch_2(H)\ch_2(H)$ &  $2/9$ & $-1/9$ & $-1/9$ \\ \hline
$\ch_3(1)\ch_2(H)$ &  $7/27$ & $-7/54$ & $-7/54$ \\ \hline
$\ch_2(H)\ch_3(1)$ &  $7/27$ & $-7/54$ & $-7/54$ \\ \hline
$\ch_3(1)\ch_3(1)$ &  $49/162$ & $-49/324$ & $-49/324$ \\ \hline
$\ch_2(\mathbf{p})$ &  $10/3$ & $-5/3$ & $-5/3$ \\ \hline
$\ch_3(H)$ &  $7/9$ & $-7/18$ & $-7/18$ \\ \hline
$\ch_4(1)$ &  $-17/18$ & $17/36$ & $17/36$ \\ \hline

  \end{tabularx}
\end{center}


\vspace{1em}
\begin{centering}
  \begin{tabularx}{\textwidth}{
      |>{\centering\arraybackslash}X
      |>{\centering\arraybackslash}X
      |>{\centering\arraybackslash}X
      |>{\centering\arraybackslash}X
      |>{\centering\arraybackslash}X |}
    \hline
    $X = \mathbb{P}^2$ & $r = 4$ & $\Delta = \mathcal{O}(1)$ & $c_2 = 3$ & $\dim M = 6$
    \\\hline
\end{tabularx}
\end{centering}

\begin{center}
\begin{tabularx}{\textwidth}{
    |>{\centering\arraybackslash}X
    |>{\centering\arraybackslash}X
    |>{\centering\arraybackslash}X |}
  \hline $F|_{X_1}$ & $F|_{X_2}$ & $F|_{X_3}$ \\ \noalign{\hrule height 1.5pt}
  $1 + s + ts^{-1} + st^{-1}$ &  $2 + t^{-1} + t$ & $1 + st^{-1} + s + s^{-1}$ \\ \hline
  $2 + ts^{-1} + st^{-1}$ & $ 1 + s^{-1} + t^{-1} + t$ & $ 1 + s + s^{-1} + t^{-1}$ \\ \hline
  $1 + t + ts^{-1} + st^{-1}$ & $ 1 + ts^{-1} + t^{-1} + t$ & $ 2 + s + s^{-1}$ \\ \hline
  $ s + t + s^2t^{-1} + t^2s^{-1}$ & $ st^{-1} + s + t + t^2s^{-1}$ & $ ts^{-1} + t + s + s^2t^{-1}$ \\ \hline
  $ s^{-1} + t^{-1} + st^{-1} + s^2t^{-1}$ & $ s + st^{-1} + t^{-1} + s^{-1}t^{-1}$ & $ t^{-1} + st^{-1} + s^{-1} t^{-1} + s^2t^{-1}$ \\ \hline
  $t^{-1} + s^{-1} + ts^{-1} + t^2s^{-1}$ & $ s^{-1} + s^{-1} t^{-1} + ts^{-1} + t^2s^{-1}$ & $ t + ts^{-1} + s^{-1} + s^{-1} t^{-1}$ \\ \hline
  $s + t + t^2 + s^2$ & $ 1 + s + ts + t^2$ & $ 1 + t + ts + s^2$ \\ \hline
  $s + t + t^2 + ts^2$ & $ 1 + t + st + st^2$ & $1 + t + st + ts^2$ \\ \hline
  $1 + st^{-1} + st^{-2} + s^2t^{-3}$ & $ 1 + st^{-1} + st^{-2} + st^{-3}$ & $ t^{-1} + st^{-1} + s^2t^{-2} + st^{-3}$ \\ \hline
  $1  + s^{-1} + ts^{-2} + t^2s^{-3}$ & $ s^{-1} + s^{-2} + ts^{-1} + t^2s^{-3}$ & $ 1 + s^{-1} + ts^{-2} + ts^{-3}$ \\ \hline
  $s + s^2 + t + st^2$ & $ 1 + s+ ts + t^2s$ & $ 1 + s+ st + s^2t$ \\ \hline
  $1 + t^{-1} + st^{-2} + s^2t^{-3}$ & $ 1 + t^{-1} + st^{-2} + st^{-3}$ & $t^{-1} + t^{-2} + st^{-1} + s^2t^{-3}$ \\ \hline
  $1 + ts^{-1} + ts^{-2} + t^2s^{-3}$ & $s^{-1} + ts^{-1} + ts^{-3} + t^2s^{-2}$ & $ 1 + ts^{-1} + ts^{-2} + ts^{-3}$ \\ \hline
\end{tabularx}
\end{center}

Note that the denominators in the next expressions are all powers of two. It is unknown to the author why this happens. For this case we added a few extra pieces of data in the end, to show how large the numbers become. Notice that the last example has a denominator of $2^{24}$. It looks very complicated, but in this case, $k = 0$, so it follows from Prop. \ref{prp:k-0--1}!

\begin{center}
  \begin{tabularx}{\textwidth}{
      |>{\centering\arraybackslash}X
      |>{\centering\arraybackslash}X
      |>{\centering\arraybackslash}X
      |>{\centering\arraybackslash}X|}
    \hline
    $D$ & $\int_M R_kD$ & $\int_M T_kD$ & $\int_M S_kD$ \\ \hline \hline
    1 &  $0$ & $-49511/4096$ & $49511/4096$ \\ \hline
$\ch_2(H)$ &  $-4227/2048$ & $3567/2048$ & $165/512$ \\ \hline
$\ch_3(1)$ &  $-63993/32768$ & $-17835/8192$ & $135333/32768$ \\ \hline
$\ch_2(H)\ch_2(H)$ &  $-43/2048$ & $3191/4096$ & $-3105/4096$ \\ \hline
$\ch_3(1)\ch_2(H)$ &  $-235/2048$ & $-15955/16384$ & $17835/16384$ \\ \hline
$\ch_2(H)\ch_3(1)$ &  $-235/2048$ & $-15955/16384$ & $17835/16384$ \\ \hline
$\ch_3(1)\ch_3(1)$ &  $10475/32768$ & $79775/65536$ & $-100725/65536$ \\ \hline
$\ch_2(\mathbf{p})$ &  $7073/1024$ & $23419/2048$ & $-37565/2048$ \\ \hline
$\ch_3(H)$ &  $-4227/2048$ & $-5751/4096$ & $14205/4096$ \\ \hline
    $\ch_4(1)$ &  $-63993/32768$ & $-376779/65536$ & $504765/65536$ \\ \hline
    $\ch_3(1)^4$ &  $-52875/131072$ & $-16875/1048576$ & $439875/1048576$ \\ \hline
    $\ch_4(1)\ch_4(1)\ch_3(1)$ &  $-9632397/8388608$ & $-2039225/2097152$ & $17789297/8388608$ \\ \hline
    $\ch_7(1)$ &  $-21331/262144$ & $-2095/196608$ & $72373/786432$ \\ \hline
    $\ch_4(1)\ch_4(1)\ch_4(1)$ &  $-34071111/8388608$ & $56785185/16777216$ & $11357037/16777216$ \\ \hline
  \end{tabularx}
\end{center}


\vspace{1em}
\begin{centering}
  \begin{tabularx}{\textwidth}{
      |>{\centering\arraybackslash}X
      |>{\centering\arraybackslash}X
      |>{\centering\arraybackslash}X
      |>{\centering\arraybackslash}X
      |>{\centering\arraybackslash}X |}
    \hline
    $X = \mathbb{P}^2$ & $r = 2$ & $\Delta = \mathcal{O}(1)$ & $c_2 = 2$ & $\dim M = 4$
    \\\hline
\end{tabularx}
\end{centering}

\begin{center}
\begin{tabularx}{.8\textwidth}{
    |>{\centering\arraybackslash}X
    |>{\centering\arraybackslash}X
    |>{\centering\arraybackslash}X |}
  \hline $F|_{X_1}$ & $F|_{X_2}$ & $F|_{X_3}$ \\ \noalign{\hrule height 1.5pt}

  $t^{-1} + 1 + ts^{-1} - t$ & $ 1 + t^{-1}$ & $ 1+s^{-1}$ \\ \hline
  $s^{-1} + 1 + st^{-1} - s$ & $ 1 + t^{-1}$ & $ 1+s^{-1}$ \\ \hline
  $t^{-1} + s^{-1}$ & $ 1 + s^{-1} + s^{-1}t^{-1} - s^{-2}$ & $ 1+s^{-1}$ \\ \hline
  $t^{-1} + s^{-1}$ & $ t^{-1} + s^{-1} + ts^{-1} - ts^{-2}$ & $ 1+s^{-1}$ \\ \hline
  $t^{-1} + s^{-1}$ & $ 1 + t^{-1}$ & $ 1 + t^{-1} + s^{-1}t^{-1} - t^{-2}$ \\ \hline
  $t^{-1} + s^{-1}$ & $ 1 + t^{-1}$ & $ s^{-1} + st^{-1} + t^{-1} - st^{-2}$ \\ \hline
  $st^{-1} + ts^{-1}$ & $ st^{-1} + t$ & $ ts^{-1} + s$ \\ \hline
  $s^{-1} + st^{-1}$ & $ t^{-1} + s$ & $ s^{-1} + s$ \\ \hline
  $t^{-1} + ts^{-1}$ & $ t + t^{-1}$ & $ t + s^{-1}$ \\ \hline
\end{tabularx}
\end{center}

\begin{center}
  \begin{tabularx}{\textwidth}{
      |>{\centering\arraybackslash}X
      |>{\centering\arraybackslash}X
      |>{\centering\arraybackslash}X
      |>{\centering\arraybackslash}X|}
    \hline
    $D$ & $\int_M R_kD$ & $\int_M T_kD$ & $\int_M S_kD$ \\ \hline \hline
    1 &  $0$ & $255/16$ & $-255/16$ \\ \hline
$\ch_2(H)$ &  $6$ & $-6$ & $0$ \\ \hline
$\ch_3(1)$ &  $-219/64$ & $15$ & $-741/64$ \\ \hline
$\ch_2(H)\ch_2(H)$ &  $-21/8$ & $27/16$ & $15/16$ \\ \hline
$\ch_3(1)\ch_2(H)$ &  $6$ & $-6$ & $0$ \\ \hline
$\ch_2(H)\ch_3(1)$ &  $6$ & $-6$ & $0$ \\ \hline
$\ch_3(1)\ch_3(1)$ &  $-27/4$ & $27/2$ & $-27/4$ \\ \hline
$\ch_2(\mathbf{p})$ &  $-51/8$ & $45/16$ & $57/16$ \\ \hline
$\ch_3(H)$ &  $6$ & $-6$ & $0$ \\ \hline
$\ch_4(1)$ &  $-219/64$ & $837/128$ & $-399/128$ \\ \hline

  \end{tabularx}
\end{center}


\vspace{1em}
\begin{centering}
  \begin{tabularx}{\textwidth}{
      |>{\centering\arraybackslash}X
      |>{\centering\arraybackslash}X
      |>{\centering\arraybackslash}X
      |>{\centering\arraybackslash}X
      |>{\centering\arraybackslash}X |}
    \hline
    $X = \mathbb{P}^2$ & $r = 2$ & $\Delta = \mathcal{O}(1)$ & $c_2 = 3$ & $\dim M = 8$
    \\\hline
\end{tabularx}
\end{centering}

\begin{center}
  \begin{longtable}[c]{|>{\centering\arraybackslash}p{.304\textwidth}
    |>{\centering\arraybackslash}p{.304\textwidth}
    |>{\centering\arraybackslash}p{.304\textwidth}|}
  \hline $F|_{X_1}$ & $F|_{X_2}$ & $F|_{X_3}$ \\ \noalign{\hrule height 1.5pt}
  \endfirsthead
$st^{-1} + 1 - s + ts^{-1} + 1 - t$ & $ 1 + t^{-1}$ & $ 1+s^{-1}$ \\ \hline
$t^{-1} + t^2s^{-1} - t^2 + 1$ & $ 1 + t^{-1}$ & $ 1+s^{-1}$ \\ \hline
$s^2t^{-1}- s^2 + 1 + s^{-1}$ & $ 1 + t^{-1}$ & $ 1+s^{-1}$ \\ \hline
$t^{-1} + ts^{-1} + s - st$ & $ 1 + t^{-1}$ & $ 1 + s^{-1}$ \\ \hline
$s^{-1} + st^{-1} + t - st$ & $ 1 + t^{-1}$ & $ 1 + s^{-1}$ \\ \hline
$t^{-1} + s^{-1}$ & $ s^{-1} + ts^{-1} - ts^{-2} + s^{-1}t^{-1} + s^{-1} - s^{-2}$ & $ 1+s^{-1}$ \\ \hline
$t^{-1} + s^{-1}$ & $ s^{-1} + t^2s^{-2} - t^2s^{-3} + t^{-1}$ & $ 1+s^{-1}$ \\ \hline
$t^{-1} + s^{-1}$ & $ 1 + t^{-1}s^{-2} + s^{-1} - s^{-3}$ & $ 1+s^{-1}$ \\ \hline
$t^{-1} + s^{-1}$ & $ 1 + ts^{-2} - ts^{-3} + s^{-1}t^{-1}$ & $ 1 + s^{-1}$ \\ \hline
$t^{-1} + s^{-1}$ & $ t^{-1} + ts^{-1} + s^{-2} - ts^{-3}$ & $ 1 + s^{-1}$ \\ \hline
$t^{-1} + s^{-1}$ & $ 1 + t^{-1}$ & $ s^2t^{-2} + t^{-1} - s^2t^{-3} +s^{-1}$ \\ \hline
$t^{-1} + s^{-1}$ & $ 1 + t^{-1}$ & $ st^{-1} + t^{-1} - st^{-2} + s^{-1}t^{-1} + t^{-1} - t^{-2}$ \\ \hline
$t^{-1} + s^{-1}$ & $ 1 + t^{-1}$ & $ 1+s^{-1}t^{-2} + t^{-1} - t^{-3}$ \\ \hline
$t^{-1} + s^{-1}$ & $ 1 + t^{-1}$ & $ 1 + s^{-1}t^{-1} + st^{-2} - st^{-3}$ \\ \hline
$t^{-1} + s^{-1}$ & $ 1 + t^{-1}$ & $ s^{-1} + t^{-2} + st^{-1} - st^{-3}$ \\ \hline
$1 + st^{-1} - s + s^{-1}$ & $ s^{-1} + ts^{-1} - ts^{-2} + t^{-1}$ & $ 1 + s^{-1}$ \\ \hline
$1 + st^{-1} - s + s^{-1}$ & $ 1 + s^{-1} t^{-1} + s^{-1} - s^{-2}$ & $ 1 + s^{-1}$ \\ \hline
$1 + st^{-1} - s + s^{-1}$ & $ 1 + t^{-1}$ & $ t^{-1} + st^{-1} - st^{-2} + s^{-1}$ \\ \hline
$1 + st^{-1} - s + s^{-1}$ & $ 1 + t^{-1}$ & $ 1 + s^{-1}t^{-1} + t^{-1} - t^{-2}$ \\ \hline
$t^{-1} + 1 + ts^{-1} - t$ & $ s^{-1} + ts^{-1} - ts^{-2} + t^{-1}$ & $ 1 + s^{-1}$ \\ \hline
$t^{-1} + 1 + ts^{-1} - t$ & $ 1 + s^{-1}t^{-1} + s^{-1} - s^{-2}$ & $ 1 + s^{-1}$ \\ \hline
$t^{-1} + 1 + ts^{-1} - t$ & $ 1 + t^{-1}$ & $ t^{-1} + st^{-1} - st^{-2} + s^{-1}$ \\ \hline
$t^{-1} + 1 + ts^{-1} - t$ & $ 1 + t^{-1}$ & $ 1 + s^{-1}t^{-1} + t^{-1} - t ^ {-2}$ \\ \hline
$t^{-1} + s^{-1}$ & $ s^{-1} + ts^{-1} - ts^{-2} + t^{-1}$ & $ t^{-1} + st^{-1} - st^{-2} + s^{-1}$ \\ \hline
$t^{-1} + s^{-1}$ & $ s^{-1} + ts^{-1} - ts^{-2} + t^{-1}$ & $ 1 + s^{-1}t^{-1} + t^{-1} - t ^ {-2}$ \\ \hline
$t^{-1} + s^{-1}$ & $ 1 + s^{-1}t^{-1} + s^{-1} - s ^ {-2}$ & $ t^{-1} + st^{-1} - st^{-2} + s^{-1}$ \\ \hline
$t^{-1} + s^{-1}$ & $ 1 + s^{-1}t^{-1} + s^{-1} - s ^ {-2}$ & $ 1 + s^{-1}t^{-1} + t^{-1} - t ^ {-2}$ \\ \hline                  
$s^2t^{-1} + s - s^2 + ts^{-1}$ & $ st^{-1} + t$ & $ ts^{-1} + s$ \\ \hline
$st^{-1} + t^2s^{-1} + t - t^2$ & $ st^{-1} + t$ & $ ts^{-1} + s$ \\ \hline
$st^{-1} + ts^{-1}$ & $ t^{-1} + 1 - s^{-1} + t$ & $ ts^{-1} + s$ \\ \hline
$st^{-1} + ts^{-1}$ & $ st^{-1} + ts^{-1} + t^2s^{-1} - t^2s^{-2}$ & $ ts^{-1} + s$ \\ \hline
$st^{-1} + ts^{-1}$ & $ st^{-1} + t$ & $ 1 - t^{-1} + s^{-1} + s$ \\ \hline
$st^{-1} + ts^{-1}$ & $ st^{-1} + t$ & $ ts^{-1} + s^2t^{-1} - s^2t^{-2} + st^{-1}$ \\ \hline
$ts^{-1} + 1 - t + st^{-1}$ & $ t^{-1} + s$ & $ s^{-1} + s$ \\ \hline
$s^{-1} + s^2t^{-1} + s - s^2$ & $ t^{-1} + s$ & $ s^{-1} + s$ \\ \hline
$s^{-1} + st^{-1}$ & $ s^{-1} + s^{-1}t^{-1} - s ^ {-2} + s$ & $ s^{-1} + s$ \\ \hline
$s^{-1} + st^{-1}$ & $ t^{-1} + t + 1 - ts^{-1}$ & $ s^{-1} + s$ \\ \hline
$s^{-1} + st^{-1}$ & $ t^{-1} + s$ & $ t^{-1} + s^{-1}t^{-1} - t^{-2} + s$ \\ \hline
$s^{-1} + st^{-1}$ & $ t^{-1} + s$ & $ s^{-1} + s^2t^{-1} + st^{-1} - s^2t^{-2}$ \\ \hline
$1 + st^{-1} - s + ts^{-1}$ & $ t + t^{-1}$ & $ t + s^{-1}$ \\ \hline
$t^{-1} + t + t^2s^{-1} - t^2$ & $ t + t^{-1}$ & $ t + s^{-1}$ \\ \hline
$t^{-1} + ts^{-1}$ & $ t^2s^{-1} + ts^{-1} - t^2s^{-2} + t^{-1}$ & $ t + s^{-1}$ \\ \hline
$t^{-1} + ts^{-1}$ & $ t + s^{-1} + s^{-1}t^{-1} - s ^ {-2}$ & $ t + s^{-1}$ \\ \hline
$t^{-1} + ts^{-1}$ & $ t + t^{-1}$ & $ s + 1 - st^{-1} + s^{-1}$ \\ \hline
$t^{-1} + ts^{-1}$ & $ t + t^{-1}$ & $ t + t^{-1} + s^{-1}t^{-1} - t^{-2}$ \\ \hline
$st^{-2} + ts^{-2}$ & $ st^{-2} + ts^{-1}$ & $ ts^{-2} + st^{-1}$ \\ \hline
$ts^{-2} + s$ & $ s^{-1} + ts$ & $ ts + ts^{-2}$ \\ \hline
$t + st^{-2}$ & $ st^{-2} + st$ & $ st + t^{-1}$ \\ \hline
\end{longtable}
\end{center}

Though the numbers in this case are not as spectacularly large as in the rank 4 case, we have added some additional cases to show that 3 and 5 can occur as factors in the denominator.

\begin{center}
  \begin{tabularx}{\textwidth}{
      |>{\centering\arraybackslash}X
      |>{\centering\arraybackslash}X
      |>{\centering\arraybackslash}X
      |>{\centering\arraybackslash}X|}
    \hline
    $D$ & $\int_M R_kD$ & $\int_M T_kD$ & $\int_M S_kD$ \\ \hline \hline
    1 &  $0$ & $6597/32$ & $-6597/32$ \\ \hline
$\ch_2(H)$ &  $291/4$ & $-291/4$ & $0$ \\ \hline
$\ch_3(1)$ &  $-8209/128$ & $2487/8$ & $-31583/128$ \\ \hline
$\ch_2(H)\ch_2(H)$ &  $-293/16$ & $5/32$ & $581/32$ \\ \hline
$\ch_3(1)\ch_2(H)$ &  $1689/16$ & $-1689/16$ & $0$ \\ \hline
$\ch_2(H)\ch_3(1)$ &  $1689/16$ & $-1689/16$ & $0$ \\ \hline
$\ch_3(1)\ch_3(1)$ &  $-11739/64$ & $15267/32$ & $-18795/64$ \\ \hline
$\ch_2(\mathbf{p})$ &  $-733/16$ & $-725/32$ & $2191/32$ \\ \hline
$\ch_3(H)$ &  $291/4$ & $-291/4$ & $0$ \\ \hline
$\ch_4(1)$ &  $-8209/128$ & $40519/256$ & $-24101/256$ \\ \hline
    $\ch_6(1)\ch_3(1)$ &  $-214651/12288$ & $8453/512$ & $11779/12288$ \\ \hline
    $\ch_6(\mathbf{p})$ &  $-733/1920$ & $1153/3840$ & $313/3840$ \\ \hline
    $\ch_8(1)$ &  $-8209/15360$ & $15757/30720$ & $661/30720$ \\ \hline
    
  \end{tabularx}
\end{center}


\vspace{1em}
\begin{centering}
  \begin{tabularx}{\textwidth}{
      |>{\centering\arraybackslash}X
      |>{\centering\arraybackslash}X
      |>{\centering\arraybackslash}X
      |>{\centering\arraybackslash}X
      |>{\centering\arraybackslash}X |}
    \hline
    $X = \mathbb{F}_0$ & $r = 2$ & $\Delta = F$ & $c_2 = 1$ & $\dim M = 1$
    \\\hline
\end{tabularx}
\end{centering}

\begin{center}
\begin{tabularx}{.8\textwidth}{
    |>{\centering\arraybackslash}X
    |>{\centering\arraybackslash}X
    |>{\centering\arraybackslash}X
    |>{\centering\arraybackslash}X |}
  \hline $F|_{X_1}$ & $F|_{X_2}$ & $F|_{X_3}$ & $F|_{X_4}$ \\ \noalign{\hrule height 1.5pt}
  $1 + t^2s^{-1}$ & $s^{-1} + t^2$ & $t^2 + 1$ & $ t^2 + 1$ \\ \hline
  $1 + t^2$ & $1 + t^2$ & $t^2 + s$ & $1 + st^2$ \\ \hline
\end{tabularx}
\end{center}

\begin{center}
  \begin{tabularx}{\textwidth}{
      |>{\centering\arraybackslash}X
      |>{\centering\arraybackslash}X
      |>{\centering\arraybackslash}X
      |>{\centering\arraybackslash}X|}
    \hline
    $D$ & $\int_M R_kD$ & $\int_M T_kD$ & $\int_M S_kD$ \\ \hline \hline
    $\ch_2(Z)$ &  $-1/2$ & $0$ & $1/2$ \\ \hline
$\ch_2(F)$ &  $1$ & $0$ & $-1$ \\ \hline
$\ch_3(1)$ &  $0$ & $0$ & $0$ \\ \hline

  \end{tabularx}
\end{center}


\vspace{1em}
\begin{centering}
  \begin{tabularx}{\textwidth}{
      |>{\centering\arraybackslash}X
      |>{\centering\arraybackslash}X
      |>{\centering\arraybackslash}X
      |>{\centering\arraybackslash}X
      |>{\centering\arraybackslash}X
      |>{\centering\arraybackslash}X |}
    \hline
    $X = \mathbb{F}_0$ & $r = 2$ & $\Delta = F + Z$ & $c_2 = 2$ & $\dim M = 3$ & $H = 2F + 5Z$
    \\\hline
\end{tabularx}
\end{centering}
\begin{center}
\begin{tabularx}{.8\textwidth}{
    |>{\centering\arraybackslash}X
    |>{\centering\arraybackslash}X
    |>{\centering\arraybackslash}X
    |>{\centering\arraybackslash}X |}
  \hline $F|_{X_1}$ & $F|_{X_2}$ & $F|_{X_3}$ & $F|_{X_4}$ \\ \noalign{\hrule height 1.5pt}
  $1 + s^{-1}t$& $t^2 + s^{-1}$& $t^2 + 1$& $t + 1$\\ \hline
 $t + t^{-1}$& $t + 1$& $s + t$& $st + t^{-1}$\\ \hline
 $t + 1$& $t^2 + 1$& $t^2 + s$& $st + 1$\\ \hline
 $s^{-1}t + t^{-1}$& $t + s^{-1}$& $t + 1$& $t + t^{-1}$\\ \hline
\end{tabularx}
\end{center}

\begin{center}
  \begin{tabularx}{\textwidth}{
      |>{\centering\arraybackslash}X
      |>{\centering\arraybackslash}X
      |>{\centering\arraybackslash}X
      |>{\centering\arraybackslash}X|}
    \hline
    $D$ & $\int_M R_kD$ & $\int_M T_kD$ & $\int_M S_kD$ \\ \hline \hline
    1 &  $0$ & $0$ & $0$ \\ \hline
$\ch_2(Z)$ &  $-1/8$ & $-1/4$ & $3/8$ \\ \hline
$\ch_2(F)$ &  $3/8$ & $3/4$ & $-9/8$ \\ \hline
$\ch_3(1)$ &  $0$ & $0$ & $0$ \\ \hline
$\ch_2(Z)\ch_2(Z)$ &  $0$ & $0$ & $0$ \\ \hline
$\ch_2(F)\ch_2(Z)$ &  $0$ & $0$ & $0$ \\ \hline
$\ch_3(1)\ch_2(Z)$ &  $-3/16$ & $0$ & $3/16$ \\ \hline
$\ch_2(Z)\ch_2(F)$ &  $0$ & $0$ & $0$ \\ \hline
$\ch_2(F)\ch_2(F)$ &  $0$ & $0$ & $0$ \\ \hline
$\ch_3(1)\ch_2(F)$ &  $9/16$ & $0$ & $-9/16$ \\ \hline
$\ch_2(Z)\ch_3(1)$ &  $-3/16$ & $0$ & $3/16$ \\ \hline
$\ch_2(F)\ch_3(1)$ &  $9/16$ & $0$ & $-9/16$ \\ \hline
$\ch_3(1)\ch_3(1)$ &  $0$ & $0$ & $0$ \\ \hline
$\ch_2(\mathbf{p})$ &  $0$ & $0$ & $0$ \\ \hline
$\ch_3(Z)$ &  $-1/8$ & $0$ & $1/8$ \\ \hline
$\ch_3(F)$ &  $3/8$ & $0$ & $-3/8$ \\ \hline
$\ch_4(1)$ &  $0$ & $0$ & $0$ \\ \hline

  \end{tabularx}
\end{center}


\vspace{1em}
\begin{centering}
  \begin{tabularx}{\textwidth}{
      |>{\centering\arraybackslash}X
      |>{\centering\arraybackslash}X
      |>{\centering\arraybackslash}X
      |>{\centering\arraybackslash}X
      |>{\centering\arraybackslash}X
      |>{\centering\arraybackslash}X |}
    \hline
    $X = \mathbb{F}_0$ & $r = 2$ & $\Delta = F + Z$ & $c_2 = 2$ & $\dim M = 3$ & $H = 5F + 2Z$
    \\\hline
\end{tabularx}
\end{centering}

\begin{center}
\begin{tabularx}{.8\textwidth}{
    |>{\centering\arraybackslash}X
    |>{\centering\arraybackslash}X
    |>{\centering\arraybackslash}X
    |>{\centering\arraybackslash}X |}
  \hline $F|_{X_1}$ & $F|_{X_2}$ & $F|_{X_3}$ & $F|_{X_4}$ \\ \noalign{\hrule height 1.5pt}
  $s + s^{-1}$& $st + s^{-1}$& $s + t$& $s + 1$\\ \hline
 $st^{-1} + s^{-1}$& $s + s^{-1}$& $s + 1$& $s + t^{-1}$\\ \hline
 $st^{-1} + 1$& $s + 1$& $s^{2} + 1$& $s^{2} + t^{-1}$\\ \hline
 $s + 1$& $st + 1$& $s^{2} + t$& $s^{2} + 1$\\ \hline
\end{tabularx}
\end{center}

\begin{center}
  \begin{tabularx}{\textwidth}{
      |>{\centering\arraybackslash}X
      |>{\centering\arraybackslash}X
      |>{\centering\arraybackslash}X
      |>{\centering\arraybackslash}X|}
    \hline
    $D$ & $\int_M R_kD$ & $\int_M T_kD$ & $\int_M S_kD$ \\ \hline \hline
    1 &  $0$ & $0$ & $0$ \\ \hline
$\ch_2(Z)$ &  $3/8$ & $3/4$ & $-9/8$ \\ \hline
$\ch_2(F)$ &  $-1/8$ & $-1/4$ & $3/8$ \\ \hline
$\ch_3(1)$ &  $0$ & $0$ & $0$ \\ \hline
$\ch_2(Z)\ch_2(Z)$ &  $0$ & $0$ & $0$ \\ \hline
$\ch_2(F)\ch_2(Z)$ &  $0$ & $0$ & $0$ \\ \hline
$\ch_3(1)\ch_2(Z)$ &  $9/16$ & $0$ & $-9/16$ \\ \hline
$\ch_2(Z)\ch_2(F)$ &  $0$ & $0$ & $0$ \\ \hline
$\ch_2(F)\ch_2(F)$ &  $0$ & $0$ & $0$ \\ \hline
$\ch_3(1)\ch_2(F)$ &  $-3/16$ & $0$ & $3/16$ \\ \hline
$\ch_2(Z)\ch_3(1)$ &  $9/16$ & $0$ & $-9/16$ \\ \hline
$\ch_2(F)\ch_3(1)$ &  $-3/16$ & $0$ & $3/16$ \\ \hline
$\ch_3(1)\ch_3(1)$ &  $0$ & $0$ & $0$ \\ \hline
$\ch_2(\mathbf{p})$ &  $0$ & $0$ & $0$ \\ \hline
$\ch_3(Z)$ &  $3/8$ & $0$ & $-3/8$ \\ \hline
$\ch_3(F)$ &  $-1/8$ & $0$ & $1/8$ \\ \hline
$\ch_4(1)$ &  $0$ & $0$ & $0$ \\ \hline
  \end{tabularx}
\end{center}


\vspace{1em}
\begin{centering}
  \begin{tabularx}{\textwidth}{
      |>{\centering\arraybackslash}X
      |>{\centering\arraybackslash}X
      |>{\centering\arraybackslash}X
      |>{\centering\arraybackslash}X
      |>{\centering\arraybackslash}X
      |>{\centering\arraybackslash}X |}
    \hline
    $X = \mathbb{F}_0$ & $r = 2$ & $\Delta = F$ & $c_2 = 2$ & $\dim M = 5$ & $H = 2F + 7Z$
    \\\hline
\end{tabularx}
\end{centering}

\begin{center}
\begin{tabularx}{.8\textwidth}{
    |>{\centering\arraybackslash}X
    |>{\centering\arraybackslash}X
    |>{\centering\arraybackslash}X
    |>{\centering\arraybackslash}X |}
  \hline $F|_{X_1}$ & $F|_{X_2}$ & $F|_{X_3}$ & $F|_{X_4}$ \\ \noalign{\hrule height 1.5pt}
 $s^{-1}t^{2} + 1$& $t^{2} + s^{-1}$& $t^{2} + 1$& $t^{2} + 1$\\ \hline
 $t^{2} + 1$& $t^{2} + 1$& $t^{2} + s$& $st^{2} + 1$\\ \hline
 $t + s^{-1}t^{2}$& $t^{3} + s^{-1}$& $t^{3} + 1$& $t^{2} + t$\\ \hline
 $t^{2} + t$& $t^{3} + 1$& $t^{3} + s$& $st^{2} + t$\\ \hline
 $s^{-1}t^{2} + t^{-1}$& $t + s^{-1}$& $t + 1$& $t^{2} + t^{-1}$\\ \hline
 $t^{2} + t^{-1}$& $t + 1$& $s + t$& $st^{2} + t^{-1}$\\ \hline
\end{tabularx}
\end{center}

\begin{center}
  \begin{tabularx}{\textwidth}{
      |>{\centering\arraybackslash}X
      |>{\centering\arraybackslash}X
      |>{\centering\arraybackslash}X
      |>{\centering\arraybackslash}X|}
    \hline
    $D$ & $\int_M R_kD$ & $\int_M T_kD$ & $\int_M S_kD$ \\ \hline \hline
    1 &  $0$ & $0$ & $0$ \\ \hline
$\ch_2(Z)$ &  $-1/32$ & $-1/8$ & $5/32$ \\ \hline
$\ch_2(F)$ &  $1/8$ & $1/2$ & $-5/8$ \\ \hline
$\ch_3(1)$ &  $0$ & $0$ & $0$ \\ \hline
$\ch_2(Z)\ch_2(Z)$ &  $0$ & $0$ & $0$ \\ \hline
$\ch_2(F)\ch_2(Z)$ &  $0$ & $0$ & $0$ \\ \hline
$\ch_3(1)\ch_2(Z)$ &  $-1/16$ & $0$ & $1/16$ \\ \hline
$\ch_2(Z)\ch_2(F)$ &  $0$ & $0$ & $0$ \\ \hline
$\ch_2(F)\ch_2(F)$ &  $0$ & $0$ & $0$ \\ \hline
$\ch_3(1)\ch_2(F)$ &  $1/4$ & $0$ & $-1/4$ \\ \hline
$\ch_2(Z)\ch_3(1)$ &  $-1/16$ & $0$ & $1/16$ \\ \hline
$\ch_2(F)\ch_3(1)$ &  $1/4$ & $0$ & $-1/4$ \\ \hline
$\ch_3(1)\ch_3(1)$ &  $0$ & $0$ & $0$ \\ \hline
$\ch_2(\mathbf{p})$ &  $0$ & $0$ & $0$ \\ \hline
$\ch_3(Z)$ &  $-1/32$ & $0$ & $1/32$ \\ \hline
$\ch_3(F)$ &  $1/8$ & $0$ & $-1/8$ \\ \hline
$\ch_4(1)$ &  $0$ & $0$ & $0$ \\ \hline
  \end{tabularx}
\end{center}


\vspace{1em}
\begin{centering}
  \begin{tabularx}{\textwidth}{
      |>{\centering\arraybackslash}X
      |>{\centering\arraybackslash}X
      |>{\centering\arraybackslash}X
      |>{\centering\arraybackslash}X
      |>{\centering\arraybackslash}X
      |>{\centering\arraybackslash}X |}
    \hline
    $X = \mathbb{F}_0$ & $r = 2$ & $\Delta = F$ & $c_2 = 2$ & $\dim M = 5$ & $H = 3F + 5Z$
    \\\hline
\end{tabularx}
\end{centering}

\begin{center}
\begin{tabularx}{.8\textwidth}{
    |>{\centering\arraybackslash}X
    |>{\centering\arraybackslash}X
    |>{\centering\arraybackslash}X
    |>{\centering\arraybackslash}X |}
  \hline $F|_{X_1}$ & $F|_{X_2}$ & $F|_{X_3}$ & $F|_{X_4}$ \\ \noalign{\hrule height 1.5pt}
  $s^{-1}t^{2} + 1$& $t^{2} + s^{-1}$& $t^{2} + 1$& $t^{2} + 1$\\ \hline
 $t^{2} + 1$& $t^{2} + 1$& $t^{2} + s$& $st^{2} + 1$\\ \hline
 $st + s + t + s^{-1}t$& $t + s^{-1}$& $t + 1$& $t + 1$\\ \hline
 $t^{2} + t + s^{-1}t^{2} + 1$& $t + s^{-1}$& $t + 1$& $t + 1$\\ \hline
 $1 + s^{-1}t$& $st - s + 1 + s^{-1}$& $t + 1$& $t + 1$\\ \hline
 $1 + s^{-1}t$& $t + 1 - t^{-1} + s^{-1}t^{-1}$& $t + 1$& $t + 1$\\ \hline
 $1 + s^{-1}t$& $t + s^{-1}$& $t + t^{-1} + s^{-1} - s^{-1}t^{-1}$& $t + 1$\\ \hline
 $1 + s^{-1}t$& $t + s^{-1}$& $1 + 1 + s^{-1}t - s^{-1}$& $t + 1$\\ \hline
 $1 + s^{-1}t$& $t + s^{-1}$& $t + 1$& $t + t - s^{-1}t + s^{-1}$\\ \hline
 $1 + s^{-1}t$& $t + s^{-1}$& $t + 1$& $t^{2} - s^{-1}t^{2} + 1 + s^{-1}t$\\ \hline
 $st + s + t + t$& $t + 1$& $s + t$& $st + 1$\\ \hline
 $st^{2} + st + t^{2} + 1$& $t + 1$& $s + t$& $st + 1$\\ \hline
 $t + 1$& $s + t - st^{-1} + t^{-1}$& $s + t$& $st + 1$\\ \hline
 $t + 1$& $st - s + 1 + 1$& $s + t$& $st + 1$\\ \hline
 $t + 1$& $t + 1$& $t + st^{-1} + 1 - t^{-1}$& $st + 1$\\ \hline
 $t + 1$& $t + 1$& $s + 1 + s^{-1}t - s^{-1}$& $st + 1$\\ \hline
 $t + 1$& $t + 1$& $s + t$& $st + t - s^{-1}t + s^{-1}$\\ \hline
 $t + 1$& $t + 1$& $s + t$& $st^{2} - t^{2} + t + 1$\\ \hline
 $s + t$& $s + t$& $s^{2} + t$& $s^{2}t + 1$\\ \hline
 $s + s^{-1}t$& $st + s^{-1}$& $st + 1$& $st + 1$\\ \hline
 $s + s^{-1}t$& $st + s^{-1}$& $s + t$& $s + t$\\ \hline
 $st + 1$& $st + 1$& $s^{2} + t$& $s^{2}t + 1$\\ \hline
\end{tabularx}
\end{center}

\begin{center}
  \begin{tabularx}{\textwidth}{
      |>{\centering\arraybackslash}X
      |>{\centering\arraybackslash}X
      |>{\centering\arraybackslash}X
      |>{\centering\arraybackslash}X|}
    \hline
    $D$ & $\int_M R_kD$ & $\int_M T_kD$ & $\int_M S_kD$ \\ \hline \hline
    1 &  $0$ & $0$ & $0$ \\ \hline
$\ch_2(Z)$ &  $-49/32$ & $-49/8$ & $245/32$ \\ \hline
$\ch_2(F)$ &  $25/8$ & $25/2$ & $-125/8$ \\ \hline
$\ch_3(1)$ &  $-1$ & $1$ & $0$ \\ \hline
$\ch_2(Z)\ch_2(Z)$ &  $-4$ & $4$ & $0$ \\ \hline
$\ch_2(F)\ch_2(Z)$ &  $2$ & $-2$ & $0$ \\ \hline
$\ch_3(1)\ch_2(Z)$ &  $-1/16$ & $-6$ & $97/16$ \\ \hline
$\ch_2(Z)\ch_2(F)$ &  $2$ & $-2$ & $0$ \\ \hline
$\ch_2(F)\ch_2(F)$ &  $8$ & $-8$ & $0$ \\ \hline
$\ch_3(1)\ch_2(F)$ &  $1/4$ & $12$ & $-49/4$ \\ \hline
$\ch_2(Z)\ch_3(1)$ &  $-1/16$ & $-6$ & $97/16$ \\ \hline
$\ch_2(F)\ch_3(1)$ &  $1/4$ & $12$ & $-49/4$ \\ \hline
$\ch_3(1)\ch_3(1)$ &  $-2$ & $2$ & $0$ \\ \hline
$\ch_2(\mathbf{p})$ &  $0$ & $0$ & $0$ \\ \hline
$\ch_3(Z)$ &  $-49/32$ & $0$ & $49/32$ \\ \hline
$\ch_3(F)$ &  $25/8$ & $0$ & $-25/8$ \\ \hline
$\ch_4(1)$ &  $-1$ & $1$ & $0$ \\ \hline

  \end{tabularx}
\end{center}


\vspace{1em}
\begin{centering}
  \begin{tabularx}{\textwidth}{
      |>{\centering\arraybackslash}X
      |>{\centering\arraybackslash}X
      |>{\centering\arraybackslash}X
      |>{\centering\arraybackslash}X
      |>{\centering\arraybackslash}X |}
    \hline
    $X = \mathbb{F}_1$ & $r = 2$ & $\Delta = F$ & $c_2 = 1$ & $\dim M = 1$
    \\\hline
\end{tabularx}
\end{centering}

\begin{center}
\begin{tabularx}{.8\textwidth}{
    |>{\centering\arraybackslash}X
    |>{\centering\arraybackslash}X
    |>{\centering\arraybackslash}X
    |>{\centering\arraybackslash}X |}
  \hline $F|_{X_1}$ & $F|_{X_2}$ & $F|_{X_3}$ & $F|_{X_4}$ \\ \noalign{\hrule height 1.5pt}
  $st + 1$& $st + 1$& $st + s$& $s^{2}t + 1$ \\ \hline
  $s + t$& $st + 1$& $st + s$& $st + s$ \\ \hline
\end{tabularx}
\end{center}

\begin{center}
  \begin{tabularx}{\textwidth}{
      |>{\centering\arraybackslash}X
      |>{\centering\arraybackslash}X
      |>{\centering\arraybackslash}X
      |>{\centering\arraybackslash}X|}
    \hline
    $D$ & $\int_M R_kD$ & $\int_M T_kD$ & $\int_M S_kD$ \\ \hline \hline
    1 &  $0$ & $0$ & $0$ \\ \hline
$\ch_2(Z)$ &  $-1$ & $0$ & $1$ \\ \hline
$\ch_2(F)$ &  $1$ & $0$ & $-1$ \\ \hline
$\ch_3(1)$ &  $0$ & $0$ & $0$ \\ \hline
  \end{tabularx}
\end{center}


\vspace{1em}
\begin{centering}
  \begin{tabularx}{\textwidth}{
      |>{\centering\arraybackslash}X
      |>{\centering\arraybackslash}X
      |>{\centering\arraybackslash}X
      |>{\centering\arraybackslash}X
      |>{\centering\arraybackslash}X |}
    \hline
    $X = \mathbb{F}_1$ & $r = 2$ & $\Delta = Z$ & $c_2 = 1$ & $\dim M = 2$
    \\\hline
 
\end{tabularx}
\end{centering}

\begin{center}
\begin{tabularx}{.8\textwidth}{
    |>{\centering\arraybackslash}X
    |>{\centering\arraybackslash}X
    |>{\centering\arraybackslash}X
    |>{\centering\arraybackslash}X |}
  \hline $F|_{X_1}$ & $F|_{X_2}$ & $F|_{X_3}$ & $F|_{X_4}$ \\ \noalign{\hrule height 1.5pt}
  $s^{2} + s$& $s^{2}t + s$& $s^{2} + st$& $s^{2} + 1$\\ \hline
 $st^{-1} + 1$& $s + 1$& $s + 1$& $s + s^{-1}t^{-1}$\\ \hline
 $s^{2} + 1$& $s^{2}t + 1$& $st + s$& $s + 1$\\ \hline
\end{tabularx}
\end{center}

\begin{center}
  \begin{tabularx}{\textwidth}{
      |>{\centering\arraybackslash}X
      |>{\centering\arraybackslash}X
      |>{\centering\arraybackslash}X
      |>{\centering\arraybackslash}X|}
    \hline
    $D$ & $\int_M R_kD$ & $\int_M T_kD$ & $\int_M S_kD$ \\ \hline \hline
        1 &  $0$ & $3/4$ & $-3/4$ \\ \hline
$\ch_2(Z)$ &  $0$ & $0$ & $0$ \\ \hline
$\ch_2(F)$ &  $0$ & $0$ & $0$ \\ \hline
$\ch_3(1)$ &  $5/16$ & $0$ & $-5/16$ \\ \hline
$\ch_2(Z)\ch_2(Z)$ &  $9/2$ & $-9/4$ & $-9/4$ \\ \hline
$\ch_2(F)\ch_2(Z)$ &  $-3/2$ & $3/4$ & $3/4$ \\ \hline
$\ch_3(1)\ch_2(Z)$ &  $0$ & $0$ & $0$ \\ \hline
$\ch_2(Z)\ch_2(F)$ &  $-3/2$ & $3/4$ & $3/4$ \\ \hline
$\ch_2(F)\ch_2(F)$ &  $1/2$ & $-1/4$ & $-1/4$ \\ \hline
$\ch_3(1)\ch_2(F)$ &  $0$ & $0$ & $0$ \\ \hline
$\ch_2(Z)\ch_3(1)$ &  $0$ & $0$ & $0$ \\ \hline
$\ch_2(F)\ch_3(1)$ &  $0$ & $0$ & $0$ \\ \hline
$\ch_3(1)\ch_3(1)$ &  $0$ & $0$ & $0$ \\ \hline
$\ch_2(\mathbf{p})$ &  $-1/2$ & $1/4$ & $1/4$ \\ \hline
$\ch_3(Z)$ &  $0$ & $0$ & $0$ \\ \hline
$\ch_3(F)$ &  $0$ & $0$ & $0$ \\ \hline
$\ch_4(1)$ &  $5/16$ & $-5/32$ & $-5/32$ \\ \hline
  \end{tabularx}
\end{center}


\vspace{1em}
\begin{centering}
  \begin{tabularx}{\textwidth}{
      |>{\centering\arraybackslash}X
      |>{\centering\arraybackslash}X
      |>{\centering\arraybackslash}X
      |>{\centering\arraybackslash}X
      |>{\centering\arraybackslash}X |}
    \hline
    $X = \mathbb{F}_1$ & $r = 2$ & $\Delta = F + Z$ & $c_2 = 1$ & $\dim M = 0$
    \\\hline
\end{tabularx}
\end{centering}

In this case Conjecture \ref{sec:strong-conjecture} is again automatic for dimension reasons.

\begin{center}
\begin{tabularx}{.8\textwidth}{
    |>{\centering\arraybackslash}X
    |>{\centering\arraybackslash}X
    |>{\centering\arraybackslash}X
    |>{\centering\arraybackslash}X |}
  \hline $F|_{X_1}$ & $F|_{X_2}$ & $F|_{X_3}$ & $F|_{X_4}$ \\ \noalign{\hrule height 1.5pt}
 $s + 1$& $st + 1$& $st + s$& $s + 1$ \\ \hline
\end{tabularx}
\end{center}


\vspace{1em}
\begin{centering}
  \begin{tabularx}{\textwidth}{
      |>{\centering\arraybackslash}X
      |>{\centering\arraybackslash}X
      |>{\centering\arraybackslash}X
      |>{\centering\arraybackslash}X
      |>{\centering\arraybackslash}X |}
    \hline
    $X = \mathbb{F}_2$ & $r = 2$ & $\Delta = F$ & $c_2 = 1$ & $\dim M = 1$
    \\\hline
\end{tabularx}
\end{centering}

\begin{center}
\begin{tabularx}{.8\textwidth}{
    |>{\centering\arraybackslash}X
    |>{\centering\arraybackslash}X
    |>{\centering\arraybackslash}X
    |>{\centering\arraybackslash}X |}
  \hline $F|_{X_1}$ & $F|_{X_2}$ & $F|_{X_3}$ & $F|_{X_4}$ \\ \noalign{\hrule height 1.5pt}
  $s^{2} + st$& $s^{2}t + s$& $s^{2}t + s^{2}$& $s^{2}t + s^{2}$\\ \hline
 $s^{2}t + 1$& $s^{2}t + 1$& $s^{2}t + s$& $s^{3}t + 1$\\ \hline
\end{tabularx}
\end{center}

\begin{center}
  \begin{tabularx}{\textwidth}{
      |>{\centering\arraybackslash}X
      |>{\centering\arraybackslash}X
      |>{\centering\arraybackslash}X
      |>{\centering\arraybackslash}X|}
    \hline
    $D$ & $\int_M R_kD$ & $\int_M T_kD$ & $\int_M S_kD$ \\ \hline \hline
    1 &  $0$ & $0$ & $0$ \\ \hline
$\ch_2(Z)$ &  $-3/2$ & $0$ & $3/2$ \\ \hline
$\ch_2(F)$ &  $1$ & $0$ & $-1$ \\ \hline
$\ch_3(1)$ &  $0$ & $0$ & $0$ \\ \hline
  \end{tabularx}
\end{center}


\vspace{1em}
\begin{centering}
  \begin{tabularx}{\textwidth}{
      |>{\centering\arraybackslash}X
      |>{\centering\arraybackslash}X
      |>{\centering\arraybackslash}X
      |>{\centering\arraybackslash}X
      |>{\centering\arraybackslash}X |}
    \hline
    $X = \mathbb{F}_2$ & $r = 2$ & $\Delta = Z$ & $c_2 = 1$ & $\dim M = 3$
    \\\hline
\end{tabularx}
\end{centering}

\begin{center}
\begin{tabularx}{.8\textwidth}{
    |>{\centering\arraybackslash}X
    |>{\centering\arraybackslash}X
    |>{\centering\arraybackslash}X
    |>{\centering\arraybackslash}X |}
  \hline $F|_{X_1}$ & $F|_{X_2}$ & $F|_{X_3}$ & $F|_{X_4}$ \\ \noalign{\hrule height 1.5pt}
  $st^{-1} + 1$& $s + 1$& $s + 1$& $s + s^{-2}t^{-1}$\\ \hline
 $s^{3} + s^{2}$& $s^{3}t + s^{2}$& $s^{3} + s^{2}t$& $s^{3} + 1$\\ \hline
 $s^{3} + s$& $s^{3}t + s$& $s^{2}t + s^{2}$& $s^{2} + 1$\\ \hline
 $s^{3} + 1$& $s^{3}t + 1$& $s^{2}t + s$& $s + 1$\\ \hline
\end{tabularx}
\end{center}

\begin{center}
  \begin{tabularx}{\textwidth}{
      |>{\centering\arraybackslash}X
      |>{\centering\arraybackslash}X
      |>{\centering\arraybackslash}X
      |>{\centering\arraybackslash}X|}
    \hline
    $D$ & $\int_M R_kD$ & $\int_M T_kD$ & $\int_M S_kD$ \\ \hline \hline
    1 &  $0$ & $0$ & $0$ \\ \hline
$\ch_2(Z)$ &  $1/2$ & $1$ & $-3/2$ \\ \hline
$\ch_2(F)$ &  $-1/8$ & $-1/4$ & $3/8$ \\ \hline
$\ch_3(1)$ &  $0$ & $0$ & $0$ \\ \hline
$\ch_2(Z)\ch_2(Z)$ &  $0$ & $0$ & $0$ \\ \hline
$\ch_2(F)\ch_2(Z)$ &  $0$ & $0$ & $0$ \\ \hline
$\ch_3(1)\ch_2(Z)$ &  $3/4$ & $0$ & $-3/4$ \\ \hline
$\ch_2(Z)\ch_2(F)$ &  $0$ & $0$ & $0$ \\ \hline
$\ch_2(F)\ch_2(F)$ &  $0$ & $0$ & $0$ \\ \hline
$\ch_3(1)\ch_2(F)$ &  $-3/16$ & $0$ & $3/16$ \\ \hline
$\ch_2(Z)\ch_3(1)$ &  $3/4$ & $0$ & $-3/4$ \\ \hline
$\ch_2(F)\ch_3(1)$ &  $-3/16$ & $0$ & $3/16$ \\ \hline
$\ch_3(1)\ch_3(1)$ &  $0$ & $0$ & $0$ \\ \hline
$\ch_2(\mathbf{p})$ &  $0$ & $0$ & $0$ \\ \hline
$\ch_3(Z)$ &  $1/2$ & $0$ & $-1/2$ \\ \hline
$\ch_3(F)$ &  $-1/8$ & $0$ & $1/8$ \\ \hline
$\ch_4(1)$ &  $0$ & $0$ & $0$ \\ \hline

  \end{tabularx}
\end{center}

\begin{center}
  \begin{tabularx}{\textwidth}{
      |>{\centering\arraybackslash}X
      |>{\centering\arraybackslash}X
      |>{\centering\arraybackslash}X
      |>{\centering\arraybackslash}X|}
    \hline
    $D$ & $\int_M R_kD$ & $\int_M T_kD$ & $\int_M S_kD$ \\ \hline \hline
  \end{tabularx}
\end{center}


\vspace{1em}
\begin{centering}
  \begin{tabularx}{\textwidth}{
      |>{\centering\arraybackslash}X
      |>{\centering\arraybackslash}X
      |>{\centering\arraybackslash}X
      |>{\centering\arraybackslash}X
      |>{\centering\arraybackslash}X |}
    \hline
    $X = \mathbb{F}_2$ & $r = 2$ & $\Delta = F + Z$ & $c_2 = 1$ & $\dim M = 1$
    \\\hline
\end{tabularx}
\end{centering}

\begin{center}
\begin{tabularx}{.8\textwidth}{
    |>{\centering\arraybackslash}X
    |>{\centering\arraybackslash}X
    |>{\centering\arraybackslash}X
    |>{\centering\arraybackslash}X |}
  \hline $F|_{X_1}$ & $F|_{X_2}$ & $F|_{X_3}$ & $F|_{X_4}$ \\ \noalign{\hrule height 1.5pt}
  $s^{2} + s$& $s^{2}t + s$& $s^{2}t + s^{2}$& $s^{2} + 1$\\ \hline
 $s^{2} + 1$& $s^{2}t + 1$& $s^{2}t + s$& $s + 1$\\ \hline
\end{tabularx}
\end{center}

\begin{center}
  \begin{tabularx}{\textwidth}{
      |>{\centering\arraybackslash}X
      |>{\centering\arraybackslash}X
      |>{\centering\arraybackslash}X
      |>{\centering\arraybackslash}X|}
    \hline
    $D$ & $\int_M R_kD$ & $\int_M T_kD$ & $\int_M S_kD$ \\ \hline \hline
    1 &  $0$ & $0$ & $0$ \\ \hline
$\ch_2(Z)$ &  $3/2$ & $0$ & $-3/2$ \\ \hline
$\ch_2(F)$ &  $-1/2$ & $0$ & $1/2$ \\ \hline
$\ch_3(1)$ &  $0$ & $0$ & $0$ \\ \hline
  \end{tabularx}
\end{center}


\printbibliography

\end{document}
